\documentclass[11pt]{article}

\usepackage{mathrsfs}

\usepackage[T1]{fontenc}
\usepackage[utf8]{inputenc}
\usepackage{lmodern}
\usepackage{amssymb,amsmath,amsthm}
\usepackage{bbm}
\usepackage{graphicx}
\usepackage{float}
\usepackage{xcolor}
\usepackage[colorlinks,linkcolor=black!70!black,citecolor=black]{hyperref} 
\usepackage[a4paper,margin=1.5cm]{geometry}
\usepackage{tikz}
\usepackage{tikz-cd}

\usetikzlibrary{decorations.pathreplacing, patterns,shapes,snakes}

\usepackage{pgfplots}

\long\def\metanote#1#2{{\color{#1}\
\ifmmode\hbox\fi{\sffamily\mdseries\upshape [#2]}\ }}

\newcommand{\uni}{{\underline{i}}}

\newcommand{\unt}{{\underline{t}}}
\newcommand{\inte}{{\text{int}}}
\newcommand{\circdelta}{{\mathring{\Delta}}}

\newcommand{\ra}{\rightarrow}

\newcommand{\be}{\begin{equation}}
\newcommand{\ee}{\end{equation}}
\newcommand{\bi}{\begin{itemize}}
\newcommand{\ei}{\end{itemize}}

\newcommand{\commentout}[1]{}

\newcommand{\Geom}{{\text{Geom}}}
\newcommand{\Span}{{\text{span}}}
\newcommand{\TV}{\text{TV}}

\newcommand{\calT}{{\mathcal{T}}}
\newcommand{\calM}{{\mathcal{M}}}

\newcommand{\Leb}{{\text{Leb}}}

\newcommand{\calI}{{\mathcal{I}}}

\newcommand{\calF}{{\mathcal{F}}}

\newcommand{\calV}{{\mathcal{V}}}
\newcommand{\calW}{{\mathcal{W}}}
\newcommand{\calP}{{\mathcal{P}}}
\newcommand{\calQ}{{\mathcal{Q}}}

\newcommand{\Nm}{{\mathbb{N}}}

\newcommand{\Rm}{{\mathbb R}}

\newcommand{\Pm}{{\mathbb P}}

\newcommand{\expE}{{\mathbb E}}
\newcommand{\Ind}{{\mathbbm{1}}}

\newtheorem{theo}{Theorem}[section]
\newtheorem{lem}[theo]{Lemma}
\newtheorem{defin}[theo]{Definition}
\newtheorem{cond}[theo]{Condition}
\newtheorem{prop}[theo]{Proposition}

\newtheorem{rmk}[theo]{Remark}

\newtheorem{conj}[theo]{Conjecture}

\newtheorem{example}[theo]{Example}
\newtheorem{claim}[theo]{Claim}
\newtheorem{observation}[theo]{Observation}

\newtheorem{Assumletter}{Assumption}[section]

\newtheorem*{rep@theorem}{\rep@title}
\newcommand{\newreptheorem}[2]{%
\newenvironment{rep#1}[1]{%
 \def\rep@title{#2 \ref{##1}}%
 \begin{rep@theorem}}%
 {\end{rep@theorem}}}
\makeatother
\newreptheorem{theorem}{Theorem}
\newreptheorem{lemma}{Lemma}
\newreptheorem{defin}{Definition}
\newreptheorem{cond}{Condition}
\newreptheorem{prop}{Proposition}
\newreptheorem{lem}{Lemma}
\newreptheorem{cor}{Corollary}
\newreptheorem{rmk}{Remark}
\newreptheorem{exer}{Exercise}
\newreptheorem{conj}{Conjecture}
\newreptheorem{assum}{Assumption}
\newreptheorem{equation}{Equation}

\long\def\metanote#1#2{{\color{#1}\
\ifmmode\hbox\fi{\sffamily\mdseries\upshape [#2]}\ }}

\begin{document}
\setcounter{page}{1}

\title{Regularity of the stationary density for systems with fast random switching}
\author{Michel Bena\"im and Oliver Tough}
\date{December 7, 2022}

\maketitle

\begin{abstract}
We consider the piecewise-deterministic Markov process obtained by randomly switching between the flows generated by a finite set of smooth vector fields on a compact set. We obtain H\"ormander-type conditions on the vector fields guaranteeing that the stationary density is: $C^k$ whenever the jump rates are sufficiently fast, for any $k<\infty$; unbounded whenever the jump rates are sufficiently slow and lower semi-continuous regardless of the jump rates. Our proofs are probabilistic, relying on a novel application of stopping times.
\end{abstract}
\section{Introduction}\label{section:introduction}

In this paper we shall consider the stationary densities of piecewise-deterministic Markov processes (PDMPs) formed by randomly switching between finitely many deterministic flows. Under general conditions, these systems are known (\cite[Theorem 1]{Bakhtin2012} and \cite[Theorem 4.5]{Benaim2015a}) to possess a unique stationary distribution, which is absolutely continuous with respect to Lebesgue measure (so they have a unique stationary density). The main goal of this paper is to establish general conditions guaranteeing that the stationary density is $C^k$ whenever the jump rate is ``sufficiently fast'', for any $k<\infty$. The requirement that the jump rates be ``sufficiently fast'' is necessary: we shall also establish broad conditions guaranteeing that the density is unbounded whenever the jump rate is ``sufficiently slow''. 

In one dimension a complete description of the regularity and formation of singularities for the stationary densities has been established in \cite{Bakhtin2015} by Bakhtin, Hurth and Mattingly. Prior to the present paper, however, much less was known in dimension greater than $1$. In \cite{Bakhtin2017}, Bakhtin, Hurth, Lawley and Mattingly considered a system consisting of two vector fields on a two-dimensional torus, with each of the two vector fields assumed to admit a smooth invariant density. In this case they established the smoothness of the invariant density. Later, in \cite{Bakhtin2021}, the same authors considered a system consisting of two linear vector fields in two dimensions. They characterised when one has boundedness of the invariant density, and when one has formation of singularities. To the authors' knowledge, prior to the present paper this was all that was known in dimension $2$ or greater. In this paper, the dimension $d<\infty$ is arbitrary.

We now describe more precisely the systems we consider in this paper. We define 
\[
E=\{1,\ldots,n\}
\]
to be a finite set of ``states'', associated to which are the $C^{\infty}(\Rm^d)$ vector fields
\[
v^i\in C^{\infty}(\Rm^d),\quad i\in E,
\]
which we assume to be bounded (we shall see this incurs no loss of generality). These vector fields define flow maps $\varphi^1_t,\ldots,\varphi^n_t$, defined by
\begin{equation}\label{eq:flow map}
\varphi^i_t(x):=x_t\quad\text{whereby}\quad x_0=x\quad\text{and}\quad \dot{x}_t=v^i(x_t)\quad\text{for}\quad t\in \Rm,\quad x\in \Rm^d,\quad i\in E.
\end{equation}

We define $M$ to be a compact subset of $\Rm^d$, which we shall always assume to be forward invariant under the flow maps: for all $i\in E$ and $x\in M$, $\varphi^i_t(x)\in M$ for all $t\geq 0$. We shall only ever be interested in the behaviour on $M$, hence it is no loss of generality to assume the vector fields $v^1,\ldots,v^n$ are bounded.

\begin{rmk}
We note that $\Rm^d$ may be replaced in this paper with any Riemannian manifold, without any changes to the theorems or proofs.
\end{rmk}

We define $C_c^k(U)$ for open subsets $U\subseteq \Rm^d$ and $0\leq k\leq \infty$ to be $k$-times continuously differentiable functions (continuous if $k=0$) on $U$ with compact support. We further define $C_0^k(U)$ for open subsets $U\subseteq \Rm^d$ and $0\leq k\leq \infty$ to be
\[
\begin{split}
C^k_0(U):=\{f\in C^k(U):\text{for every multi-index $\alpha$ with $\lvert\alpha\rvert\leq k$ and $\epsilon>0$,}\\\text{there exists a compact subset $K\subseteq U$ such that $\lvert \partial^{\alpha}f\rvert\leq \epsilon$ on $U\setminus K$}\}. 
\end{split}
\]
We define $C_0^k(U\times E)$ similarly. For a given irreducible rate matrix $Q$ on $E$, we define $(I_t)_{0\leq t<\infty}$ to be the corresponding continuous-time Markov chain on $E$ with rate matrix $Q$, so that
\[
I_t:i\mapsto j\quad\text{at rate}\quad Q_{ij}.
\]
The spatial position $X_t$ then evolves according to the ODE 
\[
\dot{X}_t=v^{I_t}(X_t),\quad X_0=x\in M,\quad t\geq 0.
\]
This defines the process $(X_t,I_t)_{0\leq t<\infty}$ on $M\times E$. Having fixed the vector fields $v^1,\ldots,v^n$, the process $(X_t,I_t)_{0\leq t<\infty}$ is determined by the choice of rate matrix $Q$. We denote as $PDMP(Q)$ the PDMP, $(X_t,I_t)_{0\leq t<\infty}$, corresponding to a given choice of $Q$. We write $\pi^Q$ for the stationary density of $PDMP(Q)$ (we shall impose a condition on the vector fields guaranteeing the existence of a unique stationary density, Assumption \ref{assum:accessible Hormander point}). Our results then provide a qualitative description for how the regularity of $\pi^Q$ depends on $Q$ (under reasonably general conditions on the vector fields). We informally summarise our results as follows:
\begin{enumerate}
\item
In Theorem \ref{theo:lower semicts stat density}, we establish that $\pi^Q$ is lower semi-continuous for any $Q$.
\item
In Theorem \ref{theo:PDMPs with fast jump rates have Ck versions}, we establish that $\pi^Q\in C^k_0(\inte(M)\times E)$ whenever the jump rates $Q$ are ``sufficiently fast''. 
\item
In Theorem \ref{theo:PDMPs with slow jump rates have unbounded density}, conditions are provided guaranteeing $\pi^Q$ possesses singularities whenever the jump rates are ``sufficiently slow''. In particular, $\pi^Q$ is guaranteed to be infinite on subsets of the domain satisfying a certain condition, Condition \ref{cond:sets giving rise to unbounded densities}, whenever the rate matrix satisfies a certain inequality, \eqref{eq:inequality for unbounded density}.
\end{enumerate}

This is qualitatively (almost) the same as the quantitative description obtained in \cite{Bakhtin2015} for the $1$-dimensional case, where the stationary density was shown to be $C^{\infty}$ when the jump rates are faster than a certain parameter, but unbounded at critical points when the jump rates are slower than a certain parameter. This isn't necessarily to be expected, due to the possibility of complex dynamical systems behaviour in dimension $d>1$.

\subsection*{Structure of the Paper}

In the following section, Section \ref{section:statement of results}, we shall provide a rigorous statement of our results. In the following section, Section \ref{section:overview of proofs}, we shall provide an overview of the proof strategy for our theorems, which employ a similar method of proof. We will also provide in this section the statement and proofs of some propositions which shall be used in the different proofs. We will then prove theorems \ref{theo:lower semicts stat density}, \ref{theo:PDMPs with fast jump rates have Ck versions} and \ref{theo:PDMPs with slow jump rates have unbounded density} in sections \ref{section proof of lower semicts stat density}, \ref{section:proof of Ck density} and \ref{section:proof of unbounded density} respectively.

\section{Statement of Results}\label{section:statement of results}

Throughout, if $f\in L^1(\Leb)$ with a version of $f$ belonging to some class of functions, we shall abuse notation by simply saying that $f$ belongs to the same class of functions. For instance, $f\in C(M)$ means that a version of $f\in L^1(\Leb)$ is continuous. Moreover, when we refer to the density of a measure, we are referring its density with respect to Lebesgue measure. Furthermore, we shall write $dx$ for $\Leb(dx)$. 

For $\uni=(i_1,\ldots,i_{\ell})\in E^{\ell}$(for any $1\leq\ell<\infty$) we define
\begin{equation}\label{eq:map Phi i defined for all vectors of pos times}
\Phi^{\uni}:M\times \Rm_{\geq 0}^{\ell}\ni (x,\unt)\mapsto \varphi^{i_{\ell}}_{t_{\ell}}\circ\ldots\circ \varphi^{i_1}_{t_1}(x)\in M,\quad \text{where}\quad \unt:=(t_1,\ldots,t_{\ell}).
\end{equation}

We shall assume throughout that there exists some $x^{\ast}\in M$ satisfying the following accessibility condition.

\begin{defin}\label{defin:accessible pt}
We say that the point $x^{\ast}$ is accessible if for every open neighbourhood $U\ni x^{\ast}$ and every $x\in M$ there exists $1\leq \ell<\infty$, $\uni\in E^{\ell}$ and $\unt\in \Rm_{\geq 0}^{\ell}$ such that $\Phi^{\uni}(x,\unt)\in U$.
\end{defin}

We define $[V,W]$ to be the Lie bracket of vector fields $V,W$. We then inductively define.
\begin{equation}\label{eq:families of Lie brackets}
\begin{split}
\mathcal{V}_0:=\{v^i:i\in E\},\quad \mathcal{V}_{n+1}:=\{[v,w]:v\in \mathcal{V}_0,\quad w\in \mathcal{V}_n\}\cup \mathcal{V}_n\quad\text{for} \quad 0\leq n<\infty,\\
\mathcal{V}_{\infty}:=\cup_{n\geq 0}\mathcal{V}_n,\quad \mathcal{V}_n(x):=\{V(x):V\in\mathcal{V}_n\}\quad\text{for}\quad x\in E\quad\text{and}\quad 0\leq n\leq \infty.
\end{split}
\end{equation}
We shall consider the following condition on points $x^{\ast}\in M$.
\begin{defin}[Weak H\"ormander condition]\label{defin:weak hormander}
We say that the weak H\"ormander condition is satisfied at $x^{\ast}\in M$ if $\text{span}(\calV_{\infty}(x^{\ast}))=\Rm^d$.
\end{defin}

We shall also consider the following stronger condition. 
\begin{defin}[$1$-H\"ormander condition]\label{defin:one bracket condition}
We say that the $1$-H\"ormander condition is satisfied at $x^{\ast}\in M$ if $\text{span}(\calV_{1}(x^{\ast}))=\Rm^d$.
\end{defin}
We impose throughout the following standing assumption. 
\renewcommand{\theAssumletter}{S}
\begin{Assumletter}\label{assum:accessible Hormander point}
There exists an accessible point $x^{\ast}\in M$ at which the weak H\"ormander condition is satisfied.
\end{Assumletter}

We define $\calQ$ to be the space of rate matrices on $E$ such that all off-diagonal entries are strictly positive (in particular, such matrices are irreducible).
\begin{rmk}
In fact, it is clear from the proofs that our results require only irreducibility, so that the statements of all of our results remain true with $\calQ$ redefined to be the space of irreducible rate matrices. This would require only technical changes to their proofs.
\end{rmk}
We note that for $Q\in\calQ$ and $\lambda>0$, the Markov chain corresponding to $\lambda Q$ is related to the Markov chain corresponding to $Q$ by multiplying the jump rates by $\lambda$, while keeping the same transition probabilities.

It is known (\cite[Theorem 1]{Bakhtin2012} and \cite[Theorem 4.5]{Benaim2015a}, see also \cite[Corollary 6.21, Theorem 6.26 and Remark 6.27]{Benaim2022}) that under Assumption \ref{assum:accessible Hormander point}, PDMP($Q$) has  a unique stationary distribution which is absolutely continuous with respect to Lebesgue measure, for all $Q\in\calQ$. We call this stationary distribution $\pi^Q$. Since it is absolutely continuous with respect to Lebesgue measure, it has a stationary density, which we also refer to as $\pi^Q$ by abuse of notation. Our first result guarantees that the stationary density is lower semicontinuous in general.
\begin{theo}\label{theo:lower semicts stat density}
Assume that Assumption \ref{assum:accessible Hormander point} is satisfied, and that $Q\in\calQ$. Then the stationary density $\pi^Q$ is a $[0,\infty]$-valued lower semi-continuous function.
\end{theo}
\begin{rmk}
The proof of Theorem \ref{theo:lower semicts stat density} does not use that the stationary distribution is absolutely continuous with respect to Lebesgue measure; we instead recover this from the proof of Theorem \ref{theo:lower semicts stat density}.
\end{rmk}

Our next result, which is the main result of this paper, provides conditions guaranteeing that the stationary density belongs to $C^k_0(\inte(M)\times E)$ whenever the jump rate is fast enough, for any $0\leq k<\infty$. We note that the notion of ``fast enough'' is not uniform in $k$, so this does not suffice to establish that the stationary density belongs to $C_0^{\infty}(\inte(M)\times E)$ for fast enough jump rates.

For any $Q\in\calQ$, the corresponding continuous time $E$-valued Markov chain $(I_t)_{0\leq t<\infty}$ has a unique stationary distribution $\varpi_Q$. This then defines the averaged vector field and averaged flow map
\begin{equation}
\bar{v}^{Q}(x):=\sum_{i\in E}\varpi_Q(i)v^i(x),\quad \bar{\varphi}^Q_s(x):=x_s\quad\text{whereby}\quad x_0=x\quad\text{and}\quad \dot{x}_s=\bar{v}^Q(x_s)\quad\text{for}\quad s\in \Rm.
\end{equation}
Since we can approximate $\bar\varphi^Q_t(x)$ arbitrarily well by switching between the flows $\varphi^1,\ldots,\varphi^n$, for all $t\geq 0$, we see that $M$ must be forward invariant under $\bar\varphi^Q$. Therefore the following is a well-defined, non-empty, compact subset of $M$ for any $Q\in\calQ$,
\begin{equation}
K^Q:=\cap_{t\geq 0}\bar{\varphi}^Q_t(M).
\end{equation}

This set is the global attractor of $\varphi^Q$ on $M$, meaning that it is compact, both forwards and backwards invariant, and attracts all points in $M$ uniformly.

We now define $\calQ_0$ to be the following set of $Q\in \calQ$,
\begin{equation}\label{eq:defin set Q0}
\calQ_0:=\{Q\in \calQ:\text{the $1$-H\"ormander condition holds at every $x\in K^Q$}\}.
\end{equation}
Note that $\calQ_0$ is necessarily a cone: if $Q\in\calQ_0$ then $\lambda Q\in\calQ_0$ for all $\lambda>0$. This is due to the fact that $\omega^Q=\omega^{\lambda Q}$, so that they have the same averaged vector field. 

\begin{rmk}
If the $1$-H\"ormander condition holds at every $x\in M$, then $\calQ_0=\calQ$.
\end{rmk}

We are now able to state the main theorem of our paper.
\begin{theo}\label{theo:PDMPs with fast jump rates have Ck versions}
We assume that Assumption \ref{assum:accessible Hormander point} holds, and that $\calQ_0$ is non-empty. Then for every $0\leq k<\infty$ there exists a locally bounded function
\begin{equation}
\Lambda_k:\calQ_0\ra [0,\infty)
\end{equation}
such that $\pi^{\lambda Q}\in C^k_0(\inte(M)\times E)$ whenever $Q\in\calQ_0$ and $\lambda\geq \Lambda_k(Q)$.
\end{theo}

We conjecture the following.
\begin{conj}\label{conj:weak Hormander sufficient}
In the assumption of Theorem \ref{theo:PDMPs with fast jump rates have Ck versions}, the $1$-H\"ormander condition can be replaced with the weak H\"ormander condition. 
\end{conj}

We shall explain in Remark \ref{rmk:generalisation of bracket condition} what one would need to prove to establish this conjecture.

Whereas the previous result established regularity of the stationary density whenever the jump rate is fast enough, the next result will guarantee that the stationary density is infinite on certain subsets of $M\times E$ whenever the jump rate is slow enough. 

\begin{cond}\label{cond:sets giving rise to unbounded densities}
The pair $(\Gamma,i)$ is such that:
\begin{enumerate}
\item
$\Gamma$ is a closed subset of $\inte(M)$ and $i\in E$.
\item
The set $\Gamma$ is backwards invariant under $\varphi^i$: if $x\in \Gamma$ then $\varphi^i_{-t}(x)\in \Gamma$ for all $t\geq 0$.
\item
There exists an open neighbourhood $N\supseteq \Gamma$ such that every point $x\in N$ is accessible.
\item
Every point $x\in \Gamma$ satisfies the weak H\"ormander condition.
\item
We have that
\begin{equation}\label{eq:uniformly negative divergence on sets gamma}
\sup_{x\in \Gamma}\text{tr}(D v^i(x))<0.
\end{equation}
\end{enumerate}
\end{cond}
Note that it is not atypical for there to exist sets satisfying Condition \ref{cond:sets giving rise to unbounded densities}. For instance, if $x\in \inte(M)$ is such that every point in an open neighbourhood of $x$ is accessible, the weak H\"ormander condition is satisfied at $x$, and the divergence of $v^i$ is negative at $x$, then $\Gamma:=\{(x,i)\}$ satisfies Condition \ref{cond:sets giving rise to unbounded densities} (see Example \ref{exam:switching between rotating and contracting vector fields}, for instance).

Given $(\Gamma,i)$ satisfying Condition \ref{cond:sets giving rise to unbounded densities}, we define
\begin{equation}\label{eq:defn of R gamma i}
R(\Gamma,i):=\inf_{x\in \Gamma}(-\text{tr}(Dv^i(x)))>0.
\end{equation}

Note that the inequality in \eqref{eq:defn of R gamma i} comes from \eqref{eq:uniformly negative divergence on sets gamma}. We recall that for $Q\in\calQ$, $Q_{ii}=-\sum_{j\neq i}Q_{ij}$, so that $-Q_{ii}>0$ is the total rate at which the corresponding Markov chain jumps from state $i$ to another state (which is positive by irreducibility). If Condition \ref{cond:sets giving rise to unbounded densities} is satisfied by some $(\Gamma,i)$, one may then consider those $Q\in \calQ$ satisfying
\begin{equation}
-Q_{ii}\leq R(\Gamma,i).
\label{eq:inequality for unbounded density}
\end{equation}

\begin{theo}\label{theo:PDMPs with slow jump rates have unbounded density}
We assume that \ref{assum:accessible Hormander point} is satisfied. We suppose that $(\Gamma,i)$ satisfies Condition \ref{cond:sets giving rise to unbounded densities} and that $Q\in\calQ$ satisfies \eqref{eq:inequality for unbounded density}. Then every lower semi-continuous version of $\pi^Q$ satisfies
\begin{equation}
\pi^Q\equiv +\infty\quad\text{on}\quad \Gamma.
\end{equation}
\end{theo}
In one dimension, \cite[Section 3.2]{Bakhtin2015} characterised when $\pi^Q$ has a singularity. Their characterisation agrees in this setting with \eqref{eq:defn of R gamma i}, so \eqref{eq:defn of R gamma i} is sharp in one dimension.

\begin{example}\label{exam:switching between rotating and contracting vector fields}
We take $d=2$ and define
\[
\begin{split}
A=\begin{pmatrix}
-1 & 1\\
-1 & -1
\end{pmatrix},\quad e_0:=\begin{pmatrix}
1\\
0
\end{pmatrix}, \quad v(x)=Ax,\quad w(x)=-(x-e_0),\\ M=\overline{\text{conv}}(\{e^{-\theta A}e_0:0\leq \theta\leq 2\pi\}).
\end{split}
\]
We see that $M$ is forwards invariant under the vector fields $v$ and $w$, which have Lie bracket 
\[
[v,w](x)=(Dw)v-(Dv)w=(-\text{Id})Ax-A(-(x-e_0))=-Ae_0. 
\]
The $1$-H\"ormander condition is satisfied everywhere on $M$ except at $x=e_0$, which cannot be contained in $K^Q$ for any $Q\in \calQ$. It follows that Assumption \ref{assum:accessible Hormander point} is satisfied and that $\calQ_0=\calQ$. Theorems \ref{theo:lower semicts stat density} and \ref{theo:PDMPs with fast jump rates have Ck versions} then imply that $\pi^Q$ is lower semi-continuous for every $Q\in \calQ$, and that for all $0\leq k<\infty$ there exists $\Lambda_k:\calQ\ra [0,\infty)$ locally bounded such that $\pi^{\lambda Q}\in C_0^k(\inte(M)\times E)$ whenever $Q\in\calQ$ and $\lambda \geq \Lambda_k(Q)$. Moreover, identifying $v$ with $i=1$ and $w$ with $i=2$, we may readily see that Condition \ref{cond:sets giving rise to unbounded densities} is satisfied by $(\{0\},1)$ (i.e. the set $\Gamma=\{0\}$ and the state $i=1$ corresponding to $v$), with $R(\{0\},1)=-\text{tr}(A)=2$. It follows that $\pi^Q((0,1))=+\infty$ whenever $Q\in\calQ$ and $-Q_{11}\leq 2$. 
\end{example}

In the survey article \cite{Malrieu2014}, the author poses the question \cite[Open Question 4]{Malrieu2014} (referring to the same PDMPs considered in this paper), \textit{``what can be said about the smoothness of the invariant density of such processes?''} The author provides an example consisting of two linear vector fields in two dimensions, \cite[Example 4.7]{Malrieu2014}, for which: 

\textit{``Thanks to the precise estimates in [3] [\cite{Bakhtin2015} in our bibliography], one can prove the following fact. If
$\lambda_0$ [$-Q_{11}$ in our notation] is small enough then, as for one-dimensional example, the density of the
invariant measure blows up at the origin.''}
\newline In precisely the same manner as in Example \ref{exam:switching between rotating and contracting vector fields}, the reader may readily verify that all of our results apply to this example.

\section{Overview of the Proofs}\label{section:overview of proofs}
The proofs of theorems \ref{theo:lower semicts stat density}, \ref{theo:PDMPs with fast jump rates have Ck versions} and \ref{theo:PDMPs with slow jump rates have unbounded density} will employ similar ideas. In particular, they all hinge on the following folklore proposition. We will provide for completeness a proof of this proposition, largely based on the proof found in \cite[Theorem 6.26]{Benaim2022}.

\begin{prop}\label{prop:relation pi tilde with omega tilde G tilde}
We suppose that $\tau$ is a stopping time, which we assume to be finite in the sense that $\Pm_{(x,i)}(\tau<\infty)=1$ for all initial conditions $(x,i)$. We may therefore define the following kernels,
\begin{equation}
S((x,i),\cdot):=\Pm_{(x,i)}((X_{\tau},I_{\tau})\in \cdot)\quad\text{and}\quad G((x,i),\cdot):=\expE_{(x,i)}\Big[\int_0^{\tau}\delta_{(X_s,I_s)}(\cdot)ds\Big],
\end{equation}
the former being a Markov kernel. 

We assume in addition that $\omega $ is a stationary distribution of $S$. Then the measure $\omega G$ is stationary for $((X_t,I_t))_{0\leq t<\infty}$. In particular, if we also have that $\expE_{\omega}[\tau]<\infty$, then the probability measure $\pi$ defined by the relation
\begin{equation}
\pi:=\frac{\omega G}{ \expE_{\omega}[\tau]}
\end{equation}
is stationary for $((X_t,I_t))_{0\leq t<\infty}$.
\end{prop}

The strategy for all of our proofs will be to construct a stopping time $\tau$ such that $\omega$ has a density belonging to some appropriate class of functions, and such that $G$ maps densities in this class to the target class for appropriate $Q$. In particular our strategies will be as follows: 
\begin{enumerate}
\item
To prove Theorem \ref{theo:lower semicts stat density}, we will construct $\tau$ such that $\omega$ always has a $C^{\infty}_0(\inte(M)\times E)$, and such that $G$ maps $C^{\infty}_0(\inte(M)\times E)$ densities to lower semi-continuous densities.
\item
To prove Theorem \ref{theo:PDMPs with fast jump rates have Ck versions}, we will construct $\tau$ such that $\omega$ always has a $C_0^{\infty}(\inte(M)\times E)$ density, and such that $G$ maps $C_0^{\infty}(\inte(M)\times E)$ densities to $C^k_0(\inte(M)\times E)$ densities whenever the jump rate is sufficiently fast.
\item
To prove Theorem \ref{theo:PDMPs with slow jump rates have unbounded density}, we will construct $\tau$ such that $\omega$ is bounded from below by a density which is continuous and strictly positive on $\Gamma$, and such that $G$ maps densities which are continuous and strictly positive on $\Gamma$ to lower semicontinuous densities which are infinite on $\Gamma$ whenever the jump rate satisfies \eqref{eq:inequality for unbounded density}.
\end{enumerate}
Whereas our results can be formulated as purely analytic statements, our proof is probabilistic and cannot be reformulated in analytic terms. This naturally leads to the question of whether there exist alternative, analytic proofs of our results? One might imagine that, if possible, this might entail some new interesting idea.

We now collect some notation and propositions which will be employed throughout the rest of the paper. In this section, we only impose Assumption \ref{assum:accessible Hormander point}.

\begin{rmk}\label{rmk:vector fields are compactly supported for proof}
For the same reason that we were able to assume without loss of generality that $v^1,\ldots,v^n$ are bounded - we are only interested in behaviour on the compact invariant set $M$ - we may similarly assume without loss of generality that $v^1,\ldots,v^n$ are compactly supported. We impose this assumption throughout all of our proofs.
\end{rmk}

Throughout this paper, for any interval $[a,b)$, we define $I_{[a,b)}$ to be the sequence of states visited over the time interval $[a,b)$, with $\calT_{[a,b)}$ the corresponding occupation times. For instance, if $I_1=2$, $I:2\ra 3$ at time $t=\frac{3}{2}$ and $I:3\ra 5$ at $t=\frac{7}{4}$, this being the final jump prior to time $2$, then $I_{[1,2)}=(2,3,5)$ and $\calT_{[1,2)}=(\frac{1}{2},\frac{1}{4},\frac{1}{4})$. We define $I_{[a,b]}$ and $\calT_{[a,b]}$ similarly.

We now introduce some notation. For $t>0$ and $\ell\in\Nm$ we define 
\[
\circdelta_t^{\ell}:=\{(t_1,\ldots,t_{\ell})\in \Rm_{>0}^{\ell}:\sum_{k=1}^{\ell}t_k<t\},\quad \Delta_t^{\ell}:=\{(t_1,\ldots,t_{\ell})\in \Rm_{>0}^{\ell}:\sum_{k=1}^{\ell}t_k=t\}. 
\]
The set $\Delta^{\ell}_t$ should be viewed as an open subset of $\Rm^{\ell-1}$ (not as a subset of $\Rm^{\ell}$, in which it's not open).

We define
\[
\begin{split}
\calI_{\ell}:=\{\underline{i}=(i_1,\ldots,i_{\ell})\in E^{\ell}_n\quad\text{for all}\quad 1\leq k<\ell\},\quad \ell\in \Nm,\quad \calI:=\cup_{\ell}\calI_{\ell}.
\end{split}
\]
For $\uni=(i_1,\ldots,i_{\ell})\in \calI$ we write $\lvert \uni\rvert:=\ell$, which we call the length of $\uni$, so that $\uni\in \calI_{\lvert \uni\rvert}$ in general.

We now suppose that $\uni=(i_1,\ldots,i_{\ell})$ is a sequence of states of length $\ell$, that $F$ is a compact subset of $M$, and that $p_{\uni}\in C_c^{\infty}(\Rm_{>0}^l;[0,1))\setminus\{0\}$. We write $\text{spt}(p_{\uni})$ for the support of $p_{\uni}$.

We define the stopping time $\tau$ as follows. We firstly define
\[
t_0:=\inf\{t>0:I_t\neq I_{t-},\quad I_{[0,t)}=\uni\},
\]
with $t_0$ defined to be $+\infty$ if no such time exists. If $t_0<\infty$, then we set $\unt:=\calT_{[0,t_0)}$. We set
\begin{equation}\label{eq:kernel S from sequence of states general section}
\tau:=\begin{cases}
t_0\quad\text{with probability $p_{\uni}(\unt)$,}\\
+\infty\quad\text{otherwise}
\end{cases}.
\end{equation}

We then define the submarkovian kernel
\begin{equation}\label{eq:kernel general section smooth density propn}
\hat{S}((x,i),\cdot):=\Pm_{(x,i)}[\{\tau<\infty\}\cap \{(X_{\tau},I_{\tau})\in \cdot\}].
\end{equation}

\begin{prop}\label{prop:general section smooth density propn}
We assume that the kernel $\hat{S}$ is given by \eqref{eq:kernel general section smooth density propn}. We further assume that for all $x\in F$,
\[
\Rm^l_{>0}\ni \unt\mapsto \Phi^{\uni}(x,\unt)\in M
\]
is a submersion at all $\unt\in \text{spt}(p_{\uni})$. Then for all $\mu\in \calP(F\times E)$, $\mu \hat{S}(\cdot)$ has a $C^{\infty}_c(\inte(M)\times E)$ density with respect to Lebesgue measure. Furthermore, for all $x'\in F$, $\unt'\in \Rm^{\ell}_{>0}$ and $j_{\ast}\in E$ such that $p_{\uni}(\unt')>0$ and $Q_{i_{\ell}j_{\ast}}>0$, there exists $c_0,r>0$ such that if $\mu\in \calP((F\cap B(x',r))\times \{i_1\})$ then 
\[
\mu \hat{S}(\cdot)\geq c_0\Leb_{\lvert_{B(\Phi^{\uni}(x',\unt'),r)}}(\cdot)\otimes \delta_{j_{\ast}}.
\]
\end{prop}

We now consider stopping times $\tau$ and events $A$ satisfying the following condition.
\begin{cond}\label{cond:condition on tau,A for propn}
The stopping time $\tau$ and event $A$ satisfies:
\begin{enumerate}
\item
The pair $(\tau,A)$ is measurable with respect to $(I_t)_{0\leq t<\infty}$ and some probability space independent of $\sigma((X_t,I_t)_{0\leq t<\infty})$.
\item
The stopping time $\tau$ is integrable for any initial condition.
\end{enumerate}
\end{cond}

Given a pair $(\tau,A)$ satisfying Condition \ref{cond:condition on tau,A for propn}, we may define the submarkovian kernels
\begin{equation}\label{eq:kernels for preservation of Ck proposition}
\hat{S}^A((x,i),\cdot):=\Pm_{(x,i)}(A\cap\{(X_{\tau},I_{\tau})\in \cdot\})\quad\text{and}\quad G((x,i),\cdot):=\expE_{(x,i)}\Big[\Ind_A\int_0^{\tau}\delta_{(X_s,I_s)}(\cdot)ds\Big].
\end{equation}
We note that the kernel defined in \eqref{eq:kernel general section smooth density propn}, which is denoted by $\hat{S}$, is a special case of the first kernel defined in \eqref{eq:kernels for preservation of Ck proposition}, denoted by $\hat{S}^A$. This may be seen by taking $A=\{\tau<\infty\}$.

\begin{prop}\label{prop:general section Ck preservation propn}
For all $0\leq k<\infty$ there exists a constant $C_k<\infty$ dependent only upon the vector fields $v^1,\ldots,v^n$ such that if $\sup_{i\in E}\expE[C_k^{\tau}\lvert I_0=i]<\infty$ and $\mu\in\calP(M\times E)$ has a $C^{\infty}_0(\inte(M)\times E)$ density, then $\mu \hat{S}$ and $\mu G$ have $C^k_0(\inte(M)\times E)$ densities.
\end{prop}

Note that in the above proposition, $C_k$ is not dependent upon the rate matrix $Q$ nor on the choice of $(A,\tau)$. We note also that if the stopping time $\tau$ is uniformly bounded, then it follows that $\mu \hat{S}$ and $\mu G$  necessarily have $C_0^{\infty}(\inte(M)\times E)$ densities.

\subsection{Proof of Proposition \ref{prop:relation pi tilde with omega tilde G tilde}}

We take some arbitrary bounded measurable function $f$. The process $((X_t,I_t))_{0\leq t<\infty}$ is defined on a filtered probability space, whose filtration we denote as $(\calF_t)_{t\geq 0}$. To shorten notation, we define $Z_t:=(X_t,I_t)$, and write $P_t$ for the transition semigroup associated to $Z_t=(X_t,I_t)$. We then have the following
\[
\begin{split}
(\omega G)(P_tf)=\expE_{\omega}\Big[\int_0^{\tau}(P_tf)(Z_s)ds\Big]=\expE_{\omega}\Big[\int_0^{\infty}(P_tf)(Z_s)\Ind(s<\tau)ds\Big]\\
=\expE_{\omega}\Big[\int_0^{\infty}\expE[f(Z_{t+s})\Ind(s<\tau)\lvert \calF_s]ds\Big]=\expE_{\omega}\Big[\int_0^{\infty}f(Z_{t+s})\Ind(s<\tau)ds\Big]\\=\expE_{\omega}\Big[\int_t^{t+\tau}f(Z_{s})ds\Big]
=\expE_{\omega}\Big[\int_{\tau}^{\tau+t}f(Z_{s})ds-\int_0^{t}f(Z_{s})ds+\int_0^{\tau}f(Z_{s})ds\Big].
\end{split}
\]
We now use the fact that $(Z_t)_{t\geq 0}$ and $(Z_{\tau+t})_{t\geq 0}$ have the same distribution (that is, the strong Markov property) to conclude that
\[
(\omega G)(P_tf)=\expE_{\omega}\Big[\int_0^{\tau}f(Z_{s})ds\Big]=(\omega G)(f).
\]
Since $f$ is arbitrary, it follows that $\omega G$ is a stationary measure for $(Z_t)_{0\leq t<\infty}$, whence we have Proposition \ref{prop:relation pi tilde with omega tilde G tilde}.
\qed

\subsection{Proof of Proposition \ref{prop:general section smooth density propn}}

We prove the following lemma
\begin{lem}\label{lem:pushforward by smooth submersions lemma}
Suppose that $U\subseteq \Rm^n$, $V\subseteq \Rm^m$ are open subsets of $\Rm^n$ and $\Rm^m$ respectively, and that $f:U\ra V$ is a $C^{\infty}$ submersion. We suppose also that $\mu$ has a $C_c^{\infty}(U)$ density, which we call $\rho$. Then the pushforward $f_{\#}\mu$ has a $C_c^{\infty}(V)$ density which is strictly positive on $f(\{x:\rho(x)>0\})$.
\end{lem}

\begin{proof}[Proof of Lemma \ref{lem:pushforward by smooth submersions lemma}]
We firstly establish the following claim.
\begin{claim}\label{claim:smooth submersion lemma holds on small balls}
For all $x\in U$, there exists an open ball $x\in B_x=B(x,r_x)\subseteq U$ such that Lemma \ref{lem:pushforward by smooth submersions lemma} holds whenever $\mu$ has a $C_c^{\infty}(B_x)$ density.
\end{claim}
\begin{proof}[Proof of Claim \ref{claim:smooth submersion lemma holds on small balls}]
We fix $x\in U$, for which we seek to establish Claim \ref{claim:smooth submersion lemma holds on small balls}. We firstly observe that if $n=m$, so that for $r_x>0$ small enough $f:B_x\ra V$ is a smooth diffeomorphism onto its image, then Claim \ref{claim:smooth submersion lemma holds on small balls} is an immediate consequence of the change of variables formula. We therefore now consider the $n>m$ case.

For all $r_x>0$ small enough, there exists linearly independent unit vectors $e_1,\ldots,e_m$ in $\Rm^n$ such that $\partial_{e_1}f(x'),\ldots,\partial_{e_m}f(x')$ are linearly independent for all $x'\in B_x$. We define $A=\text{span}(e_1,\ldots,e_m)$ and $Z=A^{\perp}$. We define $\hat{p}$ to be the orthogonal projection $\Rm^n=A\times Z\ra Z$. Therefore for all $r_x>0$ small enough,
\[
\begin{split}
g:B_x\ra V\times Z\\
x'\mapsto (f(x'),\hat{p}(x'))
\end{split}
\]
is a smooth diffeomorphism onto its image. We now fix arbitrary $x^{\ast}\in B_x$ such that $\rho(x^{\ast})>0$. We therefore have from the $n=m$ case that $g_{\#}\mu$ has a $C^{\infty}_c(V\times Z)$ density which is strictly positive at $g(x^{\ast})$.

We now define $\iota:V\times Z\ni (v,b)\mapsto v\in V$. It is then immediate that 
\[
f_{\#}\mu=\iota_{\#}(g_{\#}\mu)
\]
has a $C_c^{\infty}(V)$ density which is strictly positive at $f(x^{\ast})=\iota\circ g(x^{\ast})$. Since $x^{\ast}\in \{x'\in B_x:\rho(x')>0\}$ is arbitrary, we have established Claim \ref{claim:smooth submersion lemma holds on small balls}.
\end{proof}

We now consider for each $x\in U$ some $\varphi_x\in C_c^{\infty}(B_{x})$ such that $\Ind_{B(x,\frac{r_x}{4})}\leq \varphi_x\leq \Ind_{B(x,\frac{r_x}{2})}$. We take a finite set $x_1,\ldots,x_r$ such that $\{B(x_i,\frac{r_{x_i}}{4}):1\leq i\leq r\}$ covers $\text{spt}(\mu)$. We then inductively define
\begin{equation}\label{eq:non-neg smooth functions covering construction}
\phi^1:=\varphi_{x^1},\quad \phi^{s}:=\varphi_{x^s}\Big(1-\sum_{\ell=1}^{s-1}\phi^{\ell}\Big),\quad 2\leq s\leq r. 
\end{equation}
We see that $0\leq \phi^s\leq 1$ with $\phi^s\in C_c^{\infty}(B(x_{s},\frac{r_{x_s}}{2}))$ for all $1\leq s\leq r$, and moreover that $\sum_{s=1}^r\phi^s\equiv 1$ on $\text{spt}(\mu)$.

We define $\mu^s$ for $1\leq s\leq r$ to be the measure with density $\phi^s\rho$, to which we may apply Claim \ref{claim:smooth submersion lemma holds on small balls}. Then since $\mu=\sum_{s=1}^r\mu^s$, $f_{\#}\mu=\sum_{s=1}^rf_{\#}\mu_s$ and we are done.

\end{proof}

We now use Lemma \ref{lem:pushforward by smooth submersions lemma} to prove the following lemma.
\begin{lem}\label{lem:pushforward by smooth submersions parametrised lemma}
Suppose that $U\subseteq \Rm^n$ and $W\subseteq \Rm^k$ are open subsets of $\Rm^n$ and $\Rm^k$ respectively, that $\Upsilon\subseteq \Rm^m$ is a measurable subset of $\Rm^m$, and that $f:W\times U\ra \Upsilon$ is $C^{\infty}$. We further define $V:=\inte(\Upsilon)$ and assume that $K\subseteq W$ is a non-empty compact subset of $W$. We define $f_w:U\ni x\mapsto f(w,x)\in \Upsilon$ for all $w\in W$, which we suppose is a submersion on $\text{spt}(\mu)$ for all $w\in K$. We suppose also that $\mu$ has a $C_c^{\infty}(U)$ density, which we denote as $\rho$. Then there exists $\varphi\in C^{\infty}_c(W\times V)$ which is strictly positive on $\{(w,f(w,x)):w\in K,\rho(x)>0\}$, such that the pushforward $f^w_{\#}\mu$ has density $\frac{df^w_{\#}\mu}{d\Leb}(v)=\varphi(w,v)$ for all $w\in K$.
\end{lem}
\begin{proof}[Proof of Lemma \ref{lem:pushforward by smooth submersions parametrised lemma}]
We define the kernel
\[
H: W\ni w\mapsto f^w_{\#}\mu \in \calM(\Upsilon)\subseteq \calM(\Rm^m),
\]
and the $C^{\infty}$ map
\[
F:W\times U\ni (w,x)\mapsto (w,f(w,x))\in W\times \Upsilon.
\]
We observe that $F$ must be a submersion on $K\times \text{spt}(\mu)$, hence we can take open sets $A\supseteq K$, $B\supseteq \text{spt}(\mu)$ and $N\supset \supset A\times B$, the latter compactly containing the product of the first two, such that $F$ is a submersion on $N$. It follows that $F(N)$ is open so that $F(N)\subseteq W\times V$.

We write $\rho$ for the density of $\mu$. We then take $\psi\in C_c^{\infty}(N;\Rm_{\geq 0})$ such that $\psi((w,x))\equiv \rho(x)$ on $A\times U$, and define $\nu$ to be the measure on $W\times U$ with density $\psi$. We observe that
\begin{equation}
F_{\#}\nu(dw,dv)=H(w,dv)\Leb(dw)\quad\text{on}\quad A\times \Rm^m.
\end{equation}

It follows from Lemma \ref{lem:pushforward by smooth submersions lemma} that $F_{\#}\nu$ has a $C^{\infty}_c(W\times V)$ density which is strictly positive on $\{(w,f(w,x)):w\in K, \rho(x)>0\}$, which we denote as $\varphi(w,v)$. It therefore follows that for all $g\in C_0(W)$,
\[
Hg(w)=\int_{\Rm^m}g(v)\varphi(w,v)\Leb(dv)\quad\text{for Lebesgue almost-every $w\in A$.}
\]
Since both sides are continuous in $w\in W$, it follows that equality holds everywhere. Since $g\in C_0(\Rm^m)$ is arbitrary, it follows from the Riesz-Markov representation theorem that $f^w_{\#}\mu$ has density $(f^w_{\#}\mu)(dv)=\varphi(w,v)\Leb(dv)$ for all $w\in K$.
\end{proof}

We now use this lemma to prove Proposition \ref{prop:general section smooth density propn}. We firstly observe by construction that $\hat{S}((x,i),\cdot)$ is non-zero only if $i=i_1$. We fix arbitrary $1\leq j\leq n$ and define $\iota:\Rm^d\times\{j\}\ni (x,j)\mapsto x\in \Rm^d$. We define
\[
\tilde{S}(x,\cdot):=\iota_{\#}(\hat{S}((x,i),\cdot)_{\lvert_{E\times\{j\}}}).
\]
We take a measure $\chi\in \calP(\Rm_{>0}^{\ell})$ having a $C_c^{\infty}(\Rm_{>0}^{\ell})$ density $\rho(\unt)$ given by
\begin{equation}\label{eq:density rho proof of pushforward of measure under S has smooth density}
\rho(\unt)=p_{\uni}(\unt)\big[\prod_{k=1}^{\ell-1}Q_{i_ki_{k+1}}e^{Q_{i_ki_{k}}t_k}\big]Q_{i_{\ell}j}e^{Q_{i_{\ell}i_{\ell}}t_{\ell}}.
\end{equation}
We observe that
\[
\tilde{S}(x,\cdot)=(\unt\mapsto \Phi^{\uni}(x,\unt))_{\#}\chi.
\]
It follows from Lemma \ref{lem:pushforward by smooth submersions parametrised lemma} that there exists $\varphi\in C_c^{\infty}(\Rm^d\times \inte(M))$ such that $\tilde{S}$ satisfies $\tilde{S}(x,dy)=\varphi(x,y)\Leb(dy)$ for all $x\in F$.  It also follows from Lemma \ref{lem:pushforward by smooth submersions parametrised lemma} that $\varphi$ is strictly positive on $\{(x,\Phi^{\uni}(x,\unt):x\in F,p_{\uni}(\unt)>0\}$ whenever $Q_{i_{\ell}j}>0$. Proposition \ref{prop:general section smooth density propn} therefore follows.
\qed

\subsection{Proof of Proposition \ref{prop:general section Ck preservation propn}}

We may assume without loss of generality that $I_0=1$ with probability $1$, so that $\mu\in\calP(M\times\{1\})$.

We define $\rho$ to be the $C^{\infty}_0(\inte(M))$ density of $\mu$, so that $\mu(dy)=f(y)dy$. We define $(I_t)_{0\leq t<\infty}$ on the probability space $(\Theta,\vartheta)$, with the initial condition $X_0$ defined on a separate probability space, so that each $\theta\in\Theta$ defines a realisation of of $(I_t)_{0\leq t<\infty}$, $(I^{\theta}_t)_{0\leq t<\infty}$. We fix for the time being some arbitrary $\theta\in\Theta$. This then defines the flow map
\[
\psi^{\theta}_s(x):=x_s\quad\text{whereby}\quad \dot{x}_s:=v^{I^{\theta}_s}(x_s)\quad\text{and}\quad x_0=x.
\]
We may reverse $\psi^{\theta}_t(x)$ by defining
\[
\phi^{\theta}_t(y):=y_t\quad\text{whereby}\quad \dot{y}_s:=-v^{I^{\theta}_{t-s}}(y_s)\quad\text{and}\quad y_0=y,\quad 0\leq s\leq t.
\]
We then have that $y=\psi^{\theta}_t(x)$ if and only if $x=\phi^{\theta}_t(y)$.

Then defining $\mu^{\theta}_t:=(\psi^{\theta}_t)_{\#}\mu$, $\mu^{\theta}_t$ is given by
\[
\mu^{\theta}_t(\cdot)=\int_{M}\Ind(\psi^{\theta}_t(x)\in \cdot)\rho(x)dx=\int_{M}\Ind(y\in \cdot)\det(D\phi^{\theta}_t(y))\rho(\phi^{\theta}_t(y))dy.
\]
Therefore $\mu_t(dy)$ has density $\rho^{\theta}_t(y)$ given by
\begin{equation}\label{eq:density of pushforward measure}
\rho^{\theta}_t(y)=\det(D\phi^{\theta}_t(y))\rho(\phi^{\theta}_t(y)).
\end{equation}
We observe that $\rho^{\theta}_t\in C_0^{\infty}(\inte(M))$.

We now prove the following lemma
\begin{lem}\label{lem:derivative of flow map lemma}
For all $0\leq m<\infty$ there exists $A_m,M_m<\infty$ dependent only upon $v^1,\ldots,v^n$ such that $\lvert\lvert \varphi^{\theta}_t\rvert\rvert_{C^r}\leq A_me^{M_mt}$ for all $t\geq 0$, and for any $(I^{\theta}_s)_{0\leq s<\infty}$.
\end{lem}
\begin{proof}[Proof of Lemma \ref{lem:derivative of flow map lemma}]
We proceed by induction, noting that the $m=0$ case is immediate.

We have that
\[
D \varphi_t(x)=\int_0^tDv^{I_s}(\varphi_s(x))D\varphi_s(x)ds+\text{Id}
\]
We now define $C:=\sup_{(x,i)\in M\times E}\lvert Dv^i(x)\rvert$. Therefore there exists $A'<\infty$ such that we have
\[
\lvert\lvert D\varphi_t\rvert\rvert_{C^m}\leq \int_0^tC\lvert\lvert D\varphi_s\rvert\rvert_{C^m}ds+A'\int_0^t\lvert\lvert Dv^{I_s}\circ \varphi_s\rvert\rvert_{C^{m-1}}\lvert\lvert D\varphi_s\rvert\rvert_{C^{m-1}}ds.
\]
Applying Gronwall's inequality, we are done.
\end{proof}
We now apply Lemma \ref{lem:derivative of flow map lemma} with the reversed vector fields $-v^1,\ldots,-v^n$ to obtain $A'_{k+1},M'_{k+1}<\infty$ such that $\lvert\lvert \phi^{\theta}_t\rvert\rvert_{C^{k+1}}\leq A'_{k+1}e^{M_{k+1}'t}$ for all $\theta\in \Theta$. We then use \eqref{eq:density of pushforward measure} to see that $\lvert\lvert \rho^{\theta}_t\rvert\rvert_{C_k}\leq \tilde{A}_ke^{\tilde{M}_kt}$, for some $\tilde{A}_k,\tilde{M}_k<\infty$ dependent only upon $v^1,\ldots,v^n$.
 
We now fix arbitrary $1\leq j\leq n$. Using Fubini's theorem and the fact that $(\tau,A)$ is a measurable function of $(I^{\theta}_t)_{0\leq t<\infty}$ and some probability space independent of $(X_t,I_t)_{0\leq t<\infty}$, we see that
\[
\mu \hat{S}^A_{\lvert_{E\times \{j\}}}(\cdot)=\expE[(\varphi^{\theta}_{\tau})_{\#}\mu(\cdot)\Ind_A\Ind(I^{\theta}_{\tau}=j)]\otimes\delta_j\quad\text{and}\quad \mu G(\cdot)=\expE\Big[\int_0^{\tau}\Ind(I^{\theta}_{s}=j)(\varphi^{\theta}_{s})_{\#}\mu(\cdot)ds\Big]\otimes\delta_j.
\]
Therefore $\mu \hat{S}^A$ and $\mu G$ have densities on $E\times \{j\}$ given by
\begin{equation}\label{eq:formula for density of muS and muG}
\frac{d\mu \hat{S}^A}{d\Leb}(x)=\expE[\rho^{\theta}_{\tau}(x)\Ind_A\Ind(I^{\theta}_{\tau}=j)]\quad\text{and}\quad \frac{d\mu G}{d\Leb}(x)=\expE\Big[\int_0^{\tau}\Ind(I^{\theta}_{s}=j)\rho^{\theta}_s(x)ds\Big].
\end{equation}
Differentiating under the integral, we see that $\mu \hat{S}^A$ and $\mu G$ have $C^k$ densities with respect to Lebesgue so long as
\begin{equation}\label{eq:bound for Ck densities in prop proof}
\expE[\tilde{A}_ke^{\tilde{M}_k\tau}]<\infty\quad\text{and}\quad \expE\Big[\int_0^{\infty}\tilde{A}_ke^{\tilde{M}_kt}\Pm(\tau>t)dt\Big]<\infty.
\end{equation}

Moreover, since $\rho^{\theta}_{\tau}\in C_0^{\infty}(\inte(M))$ and $\rho^{\theta}_{s}\in C_0^{\infty}(\inte(M))$ for all $\theta$ and all $s\geq 0$, we can apply the dominated convergence theorem to \eqref{eq:formula for density of muS and muG} to see that $\mu \hat{S}^A$ and $\mu G$ have $C^k_0(\inte(M))$ densities whenever we have \eqref{eq:bound for Ck densities in prop proof}.

We therefore have Proposition \ref{prop:general section Ck preservation propn}.

\section{Proof of Theorem \ref{theo:lower semicts stat density}}\label{section proof of lower semicts stat density}

We fix $Q\in \calQ$ throughout. We fix the point $x^{\ast}\in M$ assumed to be accessible and at which the weak H\"ormander condition is satisfied. Since the weak H\"ormander condition is satisfied at $x^{\ast}$, we may choose $\epsilon,r>0$, $\xi\in \calI$ of length $\ell=\lvert \xi\rvert$ and $p\in C_c^{\infty}(\circdelta^{\ell}_{\epsilon})$ such that
\[
\circdelta_{\ell}^{\epsilon}\ni \unt\mapsto \Phi^{\xi}(x,\unt)
\]
is a submersion at all $\unt\in\text{spt}(p)$ and all $x\in B(x^{\ast},r)$, by the proof of \cite[Theorem 4.4]{Benaim2015a}.

This then defines the kernel $\hat{S}$ as in \eqref{eq:kernel general section smooth density propn}. In particular, we define
\[
t_0:=\inf\{t>0:I_t\neq I_{t-},\quad I_{[0,t)}=\xi\},
\]
with $t_0$ defined to be $+\infty$ if no such time exists. If $t_0<\infty$, then we set $\unt:=\calT_{[0, t_0)}$. We set
\begin{equation}\label{eq:hat tau stopping time proof of lower semicty}
\hat{\tau}:=\begin{cases}
t_0\quad\text{with probability $p_{\xi}(\unt)$,}\\
+\infty\quad\text{otherwise}
\end{cases}.
\end{equation}

We note in particular that if $ \hat{\tau}<\infty$ then $\hat{\tau}<\epsilon$. We define $\hat{A}=\{\hat{\tau}<\infty\}$. We then have the kernel
\begin{equation}
\hat S((x,i),\cdot):=\Pm_{(x,i)}( \hat{A}\cap \{(X_{\hat\tau},I_{\hat\tau})\in \cdot\}).
\end{equation}

We note that, reducing $r>0$ if necessary, Proposition \ref{prop:general section smooth density propn} provides us with the following lemma.
\begin{lem}\label{lem:lower semicty pf bar S doeblin Cinfty}
There exists $\hat c_0>0$ and $\nu\in\calP(M\times E)$ such that $\mu \hat S$ has a $C_0^{\infty}(\inte(M)\times E)$ density satisfying $\mu \hat S\geq \hat c_0\nu$, for all $\mu\in \calP(B(x^{\ast},r)\times \{\xi_1\})$.
\end{lem}

By the accessibility of $x^{\ast}$, for every $(x,i)\in M\times E$ there exists $\eta^{(x,i)}$ of length $\ell_{(x,i)}$, $\delta_{(x,i)}>0$, $L_{(x,i)}<\infty$ and open $U_{(x,i)}\subseteq \circdelta_{L_{(x,i)}}^{\ell_{(x,i)}}$ such that
\[
\Phi^{\eta^{(x,i)}}(x',\unt)\in B\Big(x^{\ast},\frac{r}{2}\Big)\quad\text{for every}\quad \unt\in U_{(x,i)},\quad x'\in B(x,\delta_{(x,i)}),
\]
$\eta^{(x,i)}_1=i$, and moreover $\eta^{(x,i)}_{\ell_{(x,i)}}\neq \xi_1$, so that one can jump from the final state of $\eta^{(x,i)}$ to the first state of $\xi$. We define the stopping time
$\tilde{\tau}_{(x,i)}$ and event $\tilde{A}_{(x,i)}$ as
\[
\tilde{\tau}_{(x,i)}:=\inf\{t>0:I_t=\xi_1,\quad I_{[0,t)}=\eta^{(x,i)},\quad \calT_{[0,t)}\in U_{(x,i)}\},\quad \tilde{A}_{(x,i)}:=\{\tilde{\tau}_{(x,i)}<\infty\}.
\]
We observe that if $\tilde{\tau}_{(x,i)}<\infty$, then $\tilde{\tau}_{(x,i)}< L_{(x,i)}$.

We now take $(x_1,i_1),\ldots,(x_m,i_m)$ such that $B(x_1,\frac{\delta_{(x_1,i_1)}}{2})\times \{i_1\},\ldots,B(x_m,\frac{\delta_{(x_m,i_m)}}{2})\times \{i_m\}$ is a finite cover of $M\times E$. For $1\leq k\leq m$ we write $\eta^k$, $\ell_k$, $\delta_k$, $L_k$, $U_k$, $\tilde{\tau}_k$ and $\tilde{A}_k$ for $\eta^{(x_k,i_k)}$, $\ell_{(x_k,i_k)}$, $\delta_{(x_k,i_k)}$, $L_{(x_k,i_k)}$, $U_{(x_k,i_k)}$, $\tilde{\tau}_{(x_k,i_k)}$ and $\tilde{A}_{(x_k,i_k)}$ respectively. This then defines the kernels
\[
\tilde{S}_k((x,i),\cdot):=\Pm_{(x,i)}(\tilde{A}_{k}\cap \{(X_{\tilde{\tau}_k},I_{\tilde{\tau}_k})\in \cdot\}),\quad 1\leq k\leq m.
\]
\begin{observation}\label{observation:tilde Sk maps to ball lower semicty}
We observe that for all $1\leq k\leq m$ there exists $\tilde{c}_k>0$ such that $\mu \tilde{S}_k$ is supported on $B(x^{\ast},r)\times \{\xi_1\}$ with mass at least $\tilde{c}_k$, for all $\mu\in\calP(B(x_k,\delta_k)\times \{i_k\})$. 
\end{observation}

We set 
\[
H=\sup_{1\leq j\leq m}L_k+\epsilon.
\]
We define the stopping times $\tau_k$ for $k\in \{1,\ldots,m\}$ as follows:
\begin{enumerate}
\item
We examine whether we have the event $\tilde{A}_k$. If not, we declare that $\tau_k=+\infty$. If so, we take the time $\tilde{\tau}_k$ and proceed to the next step.
\item
We then examine the time interval $[\tilde{\tau}_k,\tilde{\tau}_k+\epsilon)$, and investigate whether we have the event $\hat{A}$ on this time interval. If this is the case, we declare that $\tau_k=\tilde{\tau}_k+ \hat{\tau}$. If not, we declare that $\tau_k:=+\infty$.
\end{enumerate}
We then take $\psi^k\in C_c^{\infty}(B(x_k,\delta_k)\times \{i_k\};\Rm_{\geq 0})$ for $1\leq k\leq m$ such that $\sum_{k=1}^{\infty}\psi^k\equiv 1$ on $M\times E$, which may be constructed as in \eqref{eq:non-neg smooth functions covering construction}. This then defines the kernels
\[
R_k((x,i),\cdot)=\psi_k((x,i))\delta_{(x,i)}(\cdot),\quad 1\leq k\leq m.
\]

We can now define the stopping time $\tau$ inductively as follows. 
\begin{defin}
We investigate whether $\tau\in [nH,(n+1)H)$, beginning with $n=0$. We proceed in the following two steps:
\begin{enumerate}
\item\label{enum:step 1 stopping time construction lower semicty}
Given the initial position $(X_{nH},I_{nH})$, we choose $k\in \{1,\ldots,m\}$ with probability $\psi^k((x_{nH},I_{nH}))$.
\item\label{enum:step 2 stopping time construction lower semicty}
Having chosen $k$, we examine the time interval $[nH,(n+1)H)$. We determine whether we have the event $\{\tau_k<\infty\}$ over this time interval, where $\tau_k$ is defined as above with time $nH$ redefined to be time $0$. If this is the case, we define $\tau:=nH+\tau_k$. If not, we determine that $\tau\geq (n+1)H$.
\end{enumerate}
If it is determined that $\tau\geq (n+1)H$, we repeat the above procedure. It is easy to see that this procedure must terminate. In particular, since the probability of the process terminating in the next time $H$ is uniformly bounded away from $0$, $\tau$ must be integrable for any initial condition.
\end{defin}
\begin{rmk}\label{rmk:lower semicty filtration remark}
In steps \ref{enum:step 1 stopping time construction lower semicty} and \ref{enum:step 2 stopping time construction lower semicty} of Definition we are making random choices independently of $(X_t,I_t)_{0\leq t<\infty}$. We do this with an independent filtration $(\calF'_t)_{0\leq t<\infty}$. We then obtain a new filtration, $(\calF_t)_{0\leq t<\infty}$, by taking for $\calF_t$ the sigma-algebra generated by our original filtration at time $t$ and $\calF'_t$. Then $(X_t,I_t)_{0\leq t<\infty}$ remains Markov and $\tau$ is an $\calF$-stopping time.
\end{rmk}

We may therefore define the following kernels,
\[
\begin{split}
S((x,i),\cdot):=\Pm_{(x,i)}((X_{\tau},I_{\tau})\in \cdot),\quad G((x,i),\cdot):=\expE_{(x,i)}\Big[\int_0^{\tau}\delta_{(X_s,I_s)}(\cdot)ds\Big],\\ P((x,i),\cdot):=\Pm_{(x,i)}((X_{H},I_{H})\in \cdot,\tau> H), \quad S_0((x,i),\cdot):=\Pm_{(x,i)}((X_{\tau},I_{\tau})\in \cdot,\tau<H),\\ G_0((x,i),\cdot):=\expE_{(x,i)}\Big[\int_0^{\tau\wedge H}\delta_{(X_s,I_s)}(\cdot)ds\Big].
\end{split}
\]
We further define
\[
G_N:=\sum_{n=0}^{N}P^nG_0\quad\text{for}\quad N<\infty,\quad \text{and observe that}\quad G=\sum_{n=0}^{\infty}P^nG_0,
\]
the latter following since
\[
\begin{split}
G((x,i),\cdot)=\expE_{(x,i)}\Big[\sum_{n=0}^{\infty}\int_{\tau\wedge nH}^{\tau\wedge ((n+1)H)}\delta_{(X_s,I_s)}(\cdot)ds\Big]\\
=\sum_{n=0}^{\infty}\expE_{(x,i)}\Big[\expE\Big[\int_{nH}^{\tau\wedge ((n+1)H)}\delta_{(X_s,I_s)}(\cdot)ds\Big\lvert \calF_{nH}\Big]\Ind(\tau>nH)\Big]\\=\sum_{n=0}^{\infty}\expE_{(x,i)}\Big[G_0((X_{nH},I_{nH}),\cdot)\Ind(\tau>nH)\Big]=\sum_{n=0}^{\infty}(P^nG_0)((x,i),\cdot).
\end{split}
\]

Our goal will to prove the following lemma.
\begin{lem}\label{lem:lemma for lower semicty theorem}
There exists a unique stationary distribution for $S$, which we denote as $\omega$, and which has a $C_0^{\infty}(\inte(M)\times E)$ density. Furthermore, for all $N<\infty$, $G_N$ sends measures with $C_0^{\infty}(\inte(M)\times E)$ densities to measures with $C_0^{\infty}(\inte(M)\times E)$ densities.
\end{lem}

Before proving Lemma \ref{lem:lemma for lower semicty theorem}, we show how it provides for Theorem \ref{theo:lower semicts stat density}. 

We firstly apply Proposition \ref{prop:relation pi tilde with omega tilde G tilde} to see that
\begin{equation}\label{eq:relationship between pi and G lower semicty proof}
\pi^Q \expE_{\omega}[\tau]=\omega G.
\end{equation}
We now observe that
\[
\lvert\lvert\omega G-\omega G_N\rvert\rvert_{\TV}=\Big\lvert\Big\lvert \expE_{\omega}\Big[\int_{\tau\wedge (N+1)H}^{\tau}\delta_{(X_s,I_s)}(\cdot)ds\Big]\Big\rvert\Big\rvert_{\TV}=\expE_{\omega}[\tau-\tau\wedge (N+1)H]\ra 0\quad\text{as}\quad N\ra\infty,
\]
so that
\begin{equation}\label{eq:convergence of omega GN to omega G lower semicty proof}
\omega G_N\ra \omega G\quad\text{in total variation.}
\end{equation}

It follows from Lemma \ref{lem:lemma for lower semicty theorem} that $\omega G_N$ has a $C_0^{\infty}(\inte(M)\times E)$ density for all $N<\infty$, which we denote as $u_N$. We claim that $(u_N)_{N\in \Nm}$ is a pointwise non-decreasing sequence of continuous functions. We fix $1\leq N_1<N_2<\infty$. We observe that $\omega G_{N_1}\leq \omega G_{N_2}$ by construction. We now define $A':=\{x:u_{N_2}(x)<u_{N_1}(x)\}$. We observe that
\[
0\leq \omega G_{N_2}(A')-\omega G_{N_1}(A')=\int_{A'}(u_{N_2}-u_{N_1})((x,i))\Leb(d(x,i))\leq 0,
\]
from which we conclude that $\Leb(A')=0$. Since $u_{N_1}$ and $u_{N_2}$ are continuous, $A'=\emptyset$.

We may therefore define $u$ to be the pointwise limit
\[
u((x,i)):=\lim_{N\ra\infty}u_N((x,i))\in [0,\infty],\quad (x,i)\in M\times E,
\]
which must be lower semicontinuous (and measurable). We now fix arbitrary measurable $A''\subseteq \inte(M)\times E$. Using the monotone convergence theorem and \eqref{eq:convergence of omega GN to omega G lower semicty proof} we calculate
\[
\omega G(A'')=\lim_{N\ra\infty}\omega G_N(A'')=\lim_{N\ra \infty}\int_{A''}u_N((x,i))\Leb(d(x,i))=\int_{A''}u((x,i))\Leb(d(x,i)).
\]
It therefore follows that $\omega G$ is absolutely continuous with respect to Lebesgue, with density given by the lower semicontinuous function $u$.

We have left only to prove Lemma \ref{lem:lemma for lower semicty theorem}.
\begin{proof}[Proof of Lemma \ref{lem:lemma for lower semicty theorem}]
We firstly define for $1\leq k\leq m$,
\[
\begin{split}
A_k:=\{\tau_k<\infty\}=\{\tau_k<H\},\quad P_k((x,i),\cdot):=\Pm_{(x,i)}((X_{H},I_{H})\in \cdot,A_k^c),\\ S_{0,k}((x,i),\cdot)=\Pm_{(x,i)}((X_{\tau_k},I_{\tau_k})\in \cdot,A_k),\quad G_{0,k}((x,i),\cdot):=\expE_{(x,i)}\Big[\int_0^{\tau_k\wedge H}\delta_{(X_s,I_s)}(\cdot)ds\Big].
\end{split}
\]
We then have that
\[
S=\sum_{n=0}^{\infty}P^nS_0,\quad S_0=\sum_{k=1}^mR_kS_{0,k}\quad\text{and}\quad S_{0,k}=\tilde{S}_k\hat{S},\quad 1\leq k\leq m.
\]
We then have the following:
\begin{enumerate}
\item
We suppose that $\mu\in\calP(B(x_k,\delta_k)\times \{i_k\})$ for some $1\leq k\leq m$. It follows from Lemma \ref{lem:lower semicty pf bar S doeblin Cinfty} and Observation \ref{observation:tilde Sk maps to ball lower semicty} that there exists $c_0>0$ and $\nu\in \calP(M\times E)$, both of which are not dependent upon either $\mu$ or $k$, such that: 
\begin{enumerate}
\item
$\mu S_{0,k}\geq c_0\nu$;
\item
$\mu S_{0,k}$ has a $C_0^{\infty}(\inte(M)\times E)$ density.
\end{enumerate}
\item\label{enum:property of preserving smooth density pf of lemma for lower semicty theorem}
If $\mu\in \calP(M\times E)$ has a $C_0^{\infty}(\inte(M)\times E)$ density, then so too does $\mu P_k$ and $\mu G_{0,k}$ by Proposition \ref{prop:general section Ck preservation propn} for $1\leq k\leq m$. It is immediate that $\mu R_k$ also has a $C_0^{\infty}(\inte(M)\times E)$ density.
\end{enumerate}
Since $R_k$ only assigns mass to $B(x_k,\delta_k)\times \{i_k\}$ for $1\leq k\leq m$ and $\mu \sum_{1\leq k\leq m}R_k=\mu$ for all $\mu\in\calP(M\times E)$, we see that
\[
\mu S\geq \mu S_0\geq c_0\nu\quad\text{for all}\quad \mu\in\calP(M\times E).
\]
It therefore follows from \cite[Theorem 8.7]{Benaim2022} that $S$ has a unique stationary distribution, which we denote as $\omega$. Moroever, this stationary distribution satisfies
\[
\omega=\omega S=\sum_{k=1}^m\Big[ \Big(\sum_{n=0}^{\infty}\omega P^n\Big)R_k\Big]S_{0,k},
\]
so that $\omega$ must have a $C_0^{\infty}(\inte(M)\times E)$ density.

We now observe that
\begin{equation}\label{eq:G0 and P definition lower semicty theorem proof}
G_0=\sum_{k=1}^mR_kG_{0,k}\quad \text{and}\quad P=\sum_{k=1}^mR_kP_k.
\end{equation}
We suppose that $\mu$ has a $C_0^{\infty}(\inte(M)\times E)$ density. It follows from \eqref{eq:G0 and P definition lower semicty theorem proof} that so too does $\mu G_0$ and $\mu P$. Therefore $\mu G_N=\sum_{n\leq N}\mu P^nG_0$ must also have a $C_0^{\infty}(\inte(M)\times E)$ density for all $N<\infty$. 
\end{proof}

\section{Proof of Theorem \ref{theo:PDMPs with fast jump rates have Ck versions}}\label{section:proof of Ck density}

We fix the point $x^{\ast}\in M$ assumed to be accessible and at which the weak H\"ormander condition is satisfied. We also fix $Q_{\ast}\in \calQ_0$.

We defer for later the proof of the following proposition.

\begin{prop}\label{prop:Quantities for Ck density thm}
There exists constants $0<\delta,\epsilon,T,T'<\infty$; an open cone $Q_{\ast}\in\calQ_{\ast}\subseteq \calQ$ containing $Q_{\ast}$; an open neighbourhood $N\ni x^{\ast}$; sequence of states $\xi,\eta\in \calI$ and a non-empty open subset $U\subseteq \circdelta^{\lvert \xi\rvert}_{\epsilon}$ such that:
\begin{enumerate}
\item\label{enum:submersion on open set for correct sequence of states}
For all $x\in B(K^{Q_{\ast}},10\delta)$, the map
\begin{equation}\label{eq:submersion on U}
U\ni \unt\mapsto \Phi^{\xi}(x,\unt)\in M
\end{equation}
is a submersion. Moreover the open set $U$ satisfies
\begin{equation}\label{eq:lebesgue measure of rescaled open set intersect with simplex has pve lim inf}
\liminf_{c\ra 0}\Leb_{\circdelta_{\epsilon}^{\lvert \xi\rvert}}\Big(\frac{1}{c}U\cap\circdelta_{\epsilon}^{\lvert \xi\rvert}\Big)>0.
\end{equation}
Furthermore $\sup_j\lvert\lvert v^j\rvert \rvert_{\infty}\epsilon\leq \delta$ and $\Leb_{\circdelta_{\epsilon}^{\lvert \xi\rvert}}(\partial U)=0$.

This determines $U$, $\xi$, $\epsilon$ and $\delta$.
\item\label{enum:averaged vector field close for Q in open cone after time T}
The open cone $\calQ_{\ast}$ and constant $T<\infty$ is such that $d(\bar\varphi^{Q}_{T}(x),K^{Q_{\ast}})<\delta$ for all $x\in M$ and $Q\in \calQ_{\ast}$. Moreover $d(\bar\varphi^Q_T(x),\bar\varphi^Q_T(x^{\ast}))<\delta$ for every $x\in N$.

This determines $T$, $N$ and $Q_{\ast}$. 
\item\label{enum:access pt of propn C infty density}
For every $x\in M$, there exists $\unt\in \circdelta^{\lvert \eta\rvert}_{T'}$ such that $\Phi^{\eta}_{\unt}(x)\in N$.

This determines $\eta$ and $T'$.
\end{enumerate}
\end{prop}

\begin{rmk}\label{rmk:generalisation of bracket condition}
Recall that $\calQ_0$ is defined to be those $Q\in\calQ$ at which the $1$-H\"ormander condition is satisfied everywhere on $K^Q$. This condition is used only to establish the existence of $U$, $\xi$, $\epsilon>0$ and $\delta>0$ satisfying Part \ref{enum:submersion on open set for correct sequence of states} of Proposition \ref{prop:Quantities for Ck density thm}. Therefore in Theorem \ref{theo:PDMPs with fast jump rates have Ck versions}, the assumption that the $1$-H\"ormander condition holds everywhere on $K^Q$ can be replaced by any other assumption providing for Part \ref{enum:submersion on open set for correct sequence of states} of Proposition \ref{prop:Quantities for Ck density thm}. We conjecture that the weak H\"ormander condition holding at every $x\in K^Q$ would provide for Part \ref{enum:submersion on open set for correct sequence of states} of Proposition \ref{prop:Quantities for Ck density thm}, implying Conjecture \ref{conj:weak Hormander sufficient}.
\end{rmk}

The above proposition determines $\delta$, $\epsilon$, $T$, $T'$, $\calQ_{\ast}$, $\xi$, $\eta$ and $U$, which are henceforth fixed throughout. We further define $H=T'+T+\epsilon$. It therefore fixes the definition of the set $\Xi$ defined by the following,
\begin{equation}
\Xi:=\Xi_{T,\delta,Q^{\ast}}:=\{(\uni,\unt):\uni\in \calI,\unt\in \Delta_T^{\lvert \uni\rvert},\Phi^{\uni}(x,\unt)\in B(K^{Q^{\ast}},2\delta)\quad\text{for all}\quad x\in M\}.
\end{equation}
Furthermore, we henceforth fix $P^{\min}_{\xi}\in C_c^{\infty}(U;[0,1))\setminus \{0\}$. Note that $P^{\min}_{\xi}$ is allowed to be $0$ somewhere, but it may not be $0$ everywhere, nor may it be $1$ anywhere.

We define
\[
\calI_{\xi,r}:=\{\uni\in \calI_r:(i_{r-l+1},\ldots,i_r)=\xi\}, \quad\calI_{\xi}:=\cup_{r\geq 0}\calI_{\xi,r}.
\]
Then for every $r\in \Nm$ and every $\uni\in \calI_{\xi,r}$ we define the following open subset of $\circdelta^r_{\epsilon}$,
\begin{equation}\label{eq:defn of set Ui}
\begin{split}
U_{\uni}:=\{\unt=(t_1,\ldots,t_r)\in \circdelta^r_{\epsilon}:(t_{r-\ell+1},\ldots,t_r)\in U,\\(t_{\ell_1+1},\ldots,
t_{\ell_1+\ell})\notin \bar U\quad\text{whenever}\quad (i_{\ell_1+1},\ldots,t_{\ell_1+\ell})=\xi\quad\text{and}\quad \ell_1+\ell<r\}.
\end{split}
\end{equation}

We note that $U_{\xi}=U$.

The notion of a ``good stopping time'' is then determined by a choice of:
\begin{enumerate}
\item
A finite subset $\calI_0$ of $\calI_{\xi}$ containing $\xi$, $\xi\in\calI_0\subseteq \calI_{\xi}$.
\item
For each $\uni\in \calI_0$, some $p_{\uni}\in C_c^{\infty}(U_{\uni};[0,1))$, with the requirement that $p_{\xi}\geq P^{\min}_{\xi}$ (so that $p_{\xi}\neq 0$).
\end{enumerate}
Once we have chosen such a finite subset $\calI_0$ and choices of smooth functions $(p_i)_{i\in \calI_0}$, the definition of a ``good stopping time'' is given by the following. 
\begin{defin}[Good stopping time]
The good stopping time $\tau$ is defined inductively as follows. We investigate whether $\tau_n\in [nH,(n+1)H)$ (starting with $n=0$). For each interval $[nH,(n+1)H)$, we proceed in the following three steps:
\begin{enumerate}
\item
We examine whether $\eta$ appears as a non-terminal subsequence of $I_{[nH,nH+T')}$. In particular, we define the stopping time
\begin{equation}\label{eq:defn of stopping time t0}
\begin{split}
t_0:=\inf\{t>nH:I_t\neq I_{t-}\text{ and there exists $nH\leq t'<t$}\\
\text{such that $I_{[t',t)}=\eta$ with $I_{t'}\neq I_{t'-}$}\}<nH+T'.
\end{split}
\end{equation}
If $t_0<nH+T'$, we proceed to the next step. If $t_0\geq nH+T'$, on the other hand, then we determine that $\tau\geq (n+1)H$.

If $\eta$ does not appear as a subsequence of $I_{[nH,nH+T')}$, then we determine that $\tau\geq (n+1)H$.
\item
Having reached this step, necessarily we have that $t_0\in [nH,nH+T')$. We investigate whether 
\[
(I_{[t_0,t_0+T]},\calT_{[t_0,t_0+T]})\in \Xi.
\]
If not, we determine that $\tau\geq (n+1)H$. If so, we then proceed to the final step.
\item
We take $t'\in [t_0+T,t_0+T+\epsilon)$ such that $I_{[t_0+T,t')}\in \calI_0$, $I_{t'}\neq I_{t'-}$ and $\calT_{[t_0+T,t')}\in U_{I_{[t_0+T,t')}}$, if such a $t'$ exists (note that if such a $t'$ exists, it must be unique by the definition of $U_{\uni}$, \eqref{eq:defn of set Ui}). If such a $t'$ does not exist, then we determine that $\tau\geq (n+1)H$. If such a $t'$ does exist, we write $\uni=(i_1,\ldots,i_{\ell})=I_{[t_0+T,t')}$ and $\unt=(t_1,\ldots,t_{\ell})=\calT_{[t_0+T,t')}$. 

We then set $\tau=t'$ with probability $p_{\uni}(\unt)$, otherwise $\tau\geq (n+1)H$.
\end{enumerate}

If it is determined that $\tau\geq (n+1)H$, we repeat the above induction step. We proceed inductively in $n$ until the value of $\tau$ is determined. It is easy to see that this induction procedure must terminate, so that $\tau$ must be finite.
\end{defin}
Given a good stopping time $\tau$, we may define the following kernels,
\begin{equation}
S((x,i),\cdot):=\Pm_{(x,i)}((X_{\tau},I_{\tau})\in \cdot)\quad\text{and}\quad G((x,i),\cdot):=\expE_{(x,i)}\Big[\int_0^{\tau}\delta_{(X_s,I_s)}(\cdot)\Big].
\end{equation}
Using Proposition \ref{prop:general section smooth density propn}, we shall establish the following.
\begin{prop}\label{prop:epsilon xi for good stopping time}
For any good stopping time $\tau$, the kernel $S$ satisfies Doeblin's criterion: there exists $c_0>0$ and $\nu\in\calP(M\times E)$ such that
\[
S((x,i),\cdot)\geq c_0\nu(\cdot)\quad \text{for all}\quad (x,i)\in M\times E.
\]
Moreover $\mu S$ has a $C^{\infty}_c(\inte(M)\times E)$ density with respect to Lebesgue measure for all $\mu\in\calP(M)$.
\end{prop}
\begin{proof}[Proof of Proposition \ref{prop:epsilon xi for good stopping time}]
We define the kernels
\[
S_0((x,i),\cdot):=\Pm_{(x,i)}((X_{\tau},I_{\tau})\in \cdot,\tau<H)\quad\text{and}\quad P((x,i),\cdot):=\Pm_{(x,i)}((X_{H},I_{H})\in \cdot,\tau>H).
\]
We have that 
\[
S=(\sum_{n=0}^{\infty}P^n)S_0.
\]
Note that the mass of $\mu\sum_{n=0}^{\infty}P^n$ in bounded by the expected value of $\tau$, which is stochastically dominated by a geometric random variable, so that $\mu\sum_{n=0}^{\infty}P^n$ has finite mass. Note also that $\mu S\geq \mu S_0$. It therefore suffices to establish Proposition \ref{prop:epsilon xi for good stopping time} with $S$ replaced by $S_0$.

By Part \ref{enum:access pt of propn C infty density} of Proposition \ref{prop:Quantities for Ck density thm}, we have that
\[
\inf_{(x,i)\in M\times E}\Pm_{(x,i)}(t_0<T',(X_{t_0},I_{t_0})\in (N\cap M)\times \{1\})>0.
\]
We recall from  Part \ref{enum:averaged vector field close for Q in open cone after time T} of Proposition \ref{prop:Quantities for Ck density thm} that $d(\bar\varphi^{Q}_T(x),\bar\varphi^Q_T(x^{\ast}))<\delta$ for all $x\in N\cap M$ and $Q\in \calQ_{\ast}$. We therefore have that
\[
\inf_{(x,i)\in M\times E}\Pm_{(x,i)}(t_0<T',(X_{t_0+T},I_{t_0+T})\in (B(\bar{\varphi}_T^Q(x^{\ast}),2\delta)\cap M)\times \{\xi_1\})>0.
\]
Moreover, by construction, if $X_{t_0+T}\notin B(K^{Q^{\ast}},2\delta)$ then $\tau\geq H$. We may then apply Proposition \ref{prop:general section smooth density propn} (using the fact that $\calI_0$ is a finite set) to obtain Proposition \ref{prop:epsilon xi for good stopping time} with $S$ replaced by $S_0$.

\end{proof}

Therefore $S$ has a unique stationary distribution, $\omega$. Since $\omega$ satisfies
\[
\omega=\omega S,
\]
it follows that $\omega$ has a $C_c^{\infty}(\inte(M)\times E)$ density. We may apply Proposition \ref{prop:relation pi tilde with omega tilde G tilde} to see that
\begin{equation}
\pi\expE_{\omega}[\tau]=\omega G,
\end{equation}
for any good stopping time $\tau$. Thus if $\omega G$ has a $C^k_0(\inte(M)\times E)$ density, so too does $\pi$. Therefore, Proposition \ref{prop:general section Ck preservation propn} provides for constants $(C_k)_{k=0}^{\infty}$ such that $\pi$ has a $C^k_0(\inte(M)\times E)$ version whenever
\begin{equation}\label{eq:bound required on good stopping time for preservation of Ck by G}
\sup_{i\in E}\expE_i[C_k^{\tau}]<\infty.
\end{equation}
It therefore suffices to show that a good stopping time $\tau$ may be chosen satisfying \eqref{eq:bound required on good stopping time for preservation of Ck by G}. We will establish that such a stopping time may be chosen whenever the rate matrix is ``sufficiently fast'', essentially by application of the infinite monkey theorem.

\begin{rmk}
We adopt the convention that a geometric distribution, $\Geom(q)$, is supported on $\{1,2,\ldots\}$, so is equal to $1$ with probability $q$.
\end{rmk}

We observe that if $\inf_{i\in E}\Pm_i(\tau<H\lvert I_0=i)\geq q$, then $\tau$ is stochastically dominated by $H\text{Geom}(q)$.

In the following, when we say that a statement holds ``for all $Q$ fast enough'', we mean that there exists some locally bounded function $\Lambda:\calQ_{\ast}\ra \Rm_{>0}$ such that the statement holds for $\lambda Q$, whenever $\lambda \geq \Lambda(Q)$.

We fix $\gamma>0$. We have that $\eta$ appears as a non-terminal subsequence of $I_{[0,T')}$ (that is, $t_0<T'$) with probability at least $1-\gamma$, for all $Q$ fast enough.

We have from Part \ref{enum:averaged vector field close for Q in open cone after time T} of Proposition \ref{prop:Quantities for Ck density thm} that $d(\bar\varphi^{Q}_T(x),K^{Q_{\ast}})<\delta$ for all $Q\in\calQ_{\ast}$. Moreover we have that
\[
\Pm(\sup_{x\in M}d(\Phi^{I_{[t_0,t_0+T]}}(x,\calT_{[t_0,t_0+T]}),\bar\varphi^{Q}_T(x))<\delta)>1-\gamma\quad\text{for all $Q$ fast enough.}
\]
Therefore 
\[
(I_{[t_0,t_0+T]},\calT_{[t_0,t_0+T]})\in\xi\quad\text{with probability at least $1-\gamma$ for all $Q$ fast enough.}
\]

Following this, it is then the case by \eqref{eq:lebesgue measure of rescaled open set intersect with simplex has pve lim inf} that $\xi$ appears as a subsequence of $I_{[t_0+T,t_0+T+\epsilon)}$, with corresponding occupation times belonging to $U$, with probability at least $1-\gamma$ for all $Q$ fast enough. 

For any such $Q$ fast enough, we can therefore choose $I_0$ a finite subset of $\calI_{\xi}$ containing $\xi$, such that with probability at least $1-2\gamma$ there exists $t'\in [t_0+T,t_0+T+\epsilon)$ such that $I_{[t_0+T,t')}\in \calI_0$, $I_{t'}\neq I_{t'-}$ and $\calT_{[t_0+T,t')}\in U_{I_{[t_0+T,t')}}$. My making $p_{\uni}$ for $\uni\in \calI_0$ arbitrarily close to $\Ind_{U_{\uni}}$, we can ensure that the probability that this happens and that one then sets $\tau=t'$ is at least $1-3\gamma$.

This then ensures that for all $Q$ fast enough, there exists a good stopping time $\tau$ such that $\Pm(\tau<H\lvert I_0=i)\geq 1-5\gamma$, for any initial condition $I_0=i$. We therefore have that for any $\gamma>0$ there exists a good stopping time such that
\[
\sup_i\expE_i[C_k^{\tau}]\leq \sum_{m=1}^{\infty}(1-5\gamma)(5\gamma)^{m-1}C_k^{Hm}.
\]
Since the right hand side is finite for $\gamma>0$ small enough, there exists a good stopping time providing for \eqref{eq:bound required on good stopping time for preservation of Ck by G}.

\qed

We have left only to establish Proposition \ref{prop:Quantities for Ck density thm}.

\subsection*{Proof of Proposition \ref{prop:Quantities for Ck density thm}}

We firstly establish the existence of $U$, $\xi$, $\epsilon$ and $\delta$ providing for Part \ref{enum:submersion on open set for correct sequence of states}, before establishing the existence of $Q_{\ast}$, $T$, $N$, $\eta$ and $T'$ providing for parts \ref{enum:averaged vector field close for Q in open cone after time T} and \ref{enum:access pt of propn C infty density}.

In the following, given two sequences $a$ and $b$, we shall write $a\oplus b$ for their concatenation.

\subsection*{Part \ref{enum:submersion on open set for correct sequence of states}}

Since $\text{span}(\calV_1(x))=\Rm^d$ (recall that $\calV_1$ was defined in \eqref{eq:families of Lie brackets}) for all $x\in K^{Q_{\ast}}$, which is a compact set, there exists a bounded open neighbourhood $N\supseteq K^{Q_{\ast}}$ such that $\text{span}(\calV_1(x))=\Rm^d$ for all $x\in N$. We choose $\delta>0$ such that $B(K^{Q_{\ast}},20\delta)\subseteq N$.

We now proceed by induction. We consider the following condition on families of vector fields $\calW=\{w_1,\ldots,w_m\}$.
\begin{cond}\label{cond:condition for induction of submersion propn}
For any compact set $K\subseteq \Rm^d$ there exists:
\begin{enumerate}
\item
a sequence of states $\xi\in \calI$ of length $\ell:=\lvert\xi\rvert$;
\item
constant coefficient first-order differential operators $L^r$ ($r=1,\ldots,m$) of the form $L^r=\sum_{j=1}^{\ell}c_{rj}\partial_{t^j}$ (the $c_{rj}$ being constants);
\item
some constant $0<\epsilon<\frac{\delta}{\sup_{j}\lvert\lvert v^j \rvert\rvert_{\infty}}$ and positive functions $f_r:V\times U\ra \Rm_{>0}$ ($r=1,\ldots,m$);
\item
some family of open sets $(U_{\gamma})_{\gamma>0}$ contained in $\circdelta^{\ell}_{\epsilon}$, $U_{\gamma}\subseteq \circdelta^{\ell}_{\epsilon}$
\end{enumerate}
such that we have:
\begin{enumerate}
\item
the differential operators satisfy
\begin{equation}\label{eq:expression for diff operator applied to Phixi for induction proof}
L^r\Phi^{\xi}(x,\unt)=f_r(\unt)(w_r(x)+o_{\gamma\vee \lvert \unt\rvert}(1)),\quad\text{for all}\quad  x\in K, \; \unt\in U_{\gamma},\;1\leq r\leq m;
\end{equation}
\item
for all $\gamma>0$ the open set $U_{\gamma}$ satisfies
\begin{equation}\label{eq:positive Lebesgue lim inf for induction}
\liminf_{c\ra 0}\Leb_{\circdelta^{\ell}_{\epsilon}}(\frac{1}{c}U_{\gamma}\cap \circdelta^{\epsilon}_{\ell})>0\quad\text{and}\quad \Leb_{\circdelta^{\ell}_{\epsilon}}(\partial U_{\gamma})=0.
\end{equation}
\end{enumerate}
\end{cond}
\begin{rmk}
The $o$ in \eqref{eq:expression for diff operator applied to Phixi for induction proof} should be interpreted as being uniform over all $x\in K$. In the following, whenever we take some compact set for which we seek to prove some family of vector fields satisfied Condition \ref{cond:condition for induction of submersion propn}, any $o$ we use should be understood to be uniform over all $x$ in the given compact set.
\end{rmk}

Our goal then is to show that $\calV_1$ satisfies Condition \ref{cond:condition for induction of submersion propn}, which we apply to $K=\bar N$. We would then obtain $\xi\in\calI$ of length $\ell=\lvert \xi\rvert$, $0<\epsilon<\frac{\delta}{\sup_{j}\lvert\lvert v^j \rvert\rvert_{\infty}}$ and open subsets $(U_{\gamma})_{\gamma>0}$ of $\circdelta_{\ell}^{\epsilon}$ satisfying \eqref{eq:positive Lebesgue lim inf for induction} such that
\[
V(x)+o_{\gamma\vee \lvert \unt\rvert}(1)\in \Span(D_{\unt}\Phi^{\xi}(x,\unt))\quad\text{for all}\quad x\in K,\;\unt\in U_{\gamma},\;V\in \calV_1.
\]
Therefore $\Span(D_{\unt}\Phi^{\xi}(x,\unt))=\Rm^d$ for all $x\in B(K^{Q_{\ast}},10\delta)$ and $\unt\in U^{\gamma}$ such that $\lvert \unt\rvert\vee \gamma$ is small enough. Note that we can make $\epsilon$ as small as we like, by replacing $U_{\gamma}$ with $\circdelta_{\ell}^{\epsilon}\cap U_{\gamma}$, ensuring that $\lvert \unt\rvert$ is sufficiently small. Then taking this $\xi\in \calI$, $\delta>0$ and $U=U_{\gamma}\subseteq\circdelta_{\ell}^{\epsilon}$ for $\gamma\vee\epsilon>0$ sufficiently small, we obtain Part \ref{enum:submersion on open set for correct sequence of states} of Proposition \ref{prop:Quantities for Ck density thm}.

It therefore suffices to show that $\calV_1$ satisfies Condition \ref{cond:condition for induction of submersion propn}, which we now establish by induction.

\begin{lem}\label{lem:union of two families of vector fields satisfies submersion condition}
Suppose that $\calW^1$ and $\calW^2$ satisfies Condition \ref{cond:condition for induction of submersion propn}. Then $\calW:=\calW^1\cup\calW^2$ satisfies Condition \ref{cond:condition for induction of submersion propn}.
\end{lem}

\begin{proof}[Proof of Lemma \ref{lem:union of two families of vector fields satisfies submersion condition}]
We fix an arbitrary compact set $K\subseteq \Rm^d$ for which we seek to verify $\calW$ satisfies Condition \ref{cond:condition for induction of submersion propn}. We then take some larger compact set $K'$ containing $K$ in its interior, $\inte(K')\supseteq K$.

Corresponding to this $K'$, we take $\xi^i$ of length $\ell_i$, $L^{i,r}=\sum_{j=1}^{\ell_i}c^i_{rj}\partial_{t^j}$, $(U^i_{\gamma})_{\gamma>0}$, $\epsilon_i$ and $f^i_r$ providing for Condition \ref{cond:condition for induction of submersion propn} to be satisfied by $\calW^i=(w^i_1,\ldots,w^i_{m_i})$, for $i=1,2$. We then define $\xi:=\xi_1\oplus\xi_2$, which is of length $\ell:=\ell^1+\ell^2$. We define the differential operators $L^r$ ($1\leq r\leq m:=m_1+m_2)$ to be
\[
L^r=\sum_{j=1}^{\ell_1}c_{rj}\partial_{t^j},\quad 1\leq r\leq m_1,\quad L^r=\sum_{j=1}^{\ell_2}c_{(r-m_1)j}\partial_{t^{\ell_1+j}},\quad m_1+1\leq r\leq m_2.
\]
We further define $\epsilon=\epsilon_1\wedge \epsilon_2\wedge \frac{d(\partial K,K')}{\sup_{1\leq j\leq n}\lvert\lvert v^j \rvert\rvert_{\infty}}$ and $U_{\gamma}:=U^1_{\gamma}\times U^2_{\gamma}$. Note that the former ensures that if $x\in K'$, then we cannot leave $K$ in time $\epsilon$. Note also that $\epsilon\leq \epsilon_1<\frac{\delta}{\sup_j\lvert\lvert v^j \rvert\rvert_{\infty}}$

For $\unt\in U$ we write $\unt=(\unt_1,\unt_2)$ whereby $\unt_1\in U_1$ and $\unt_2\in U_2$. For $1\leq r\leq m_1$ we calculate
\[
L^r\Phi^{\xi}(x,\unt)=D_x\Phi^{\xi_2}(\Phi^{\xi_1}(x,\unt_1),\unt_2)L^{1,r}\Phi^{\xi_1}(x,\unt_1)=(\text{Id}+o(\lvert \unt_2\rvert))(f^1_r(\unt_1)(w^1_r(x)+o_{\gamma_1\vee \lvert \unt_1\rvert}(1))),
\]
which is of the form \eqref{eq:expression for diff operator applied to Phixi for induction proof} when we set $f_r(\unt)=f^1_r(\unt_1)$.

We now consider $m_1+1\leq r\leq m_1+m_2$. Using the fact that $x\in K'$ implies that $\Phi^{\xi_1}(x,\unt_1)\in K$, we calculate
\[
L^r\Phi^{\xi}(x,\unt)=L^{2,r-m_1}\Phi^{\xi_2}(\Phi^{\xi_1}(x,\unt_1),\unt_2)=f^2_{r-m_1}(\unt_2)(w_{r-m_1}^2(\Phi^{\xi_1}(x,\unt_1))+o_{\gamma_2+\lvert \unt_2\rvert}(1)),
\]
which is of the form \eqref{eq:expression for diff operator applied to Phixi for induction proof} when we set $f_r(\unt)=f^2_{r-m_1}(\unt_2)$. Therefore we have \eqref{eq:expression for diff operator applied to Phixi for induction proof} for all $x\in K'$, $\unt\in U_{\gamma}$ and $v\in \calW$.

Moreover, it is immediate that $U_{\gamma}$ satisfies \eqref{eq:positive Lebesgue lim inf for induction} for all $\gamma>0$, so we are done.
\end{proof}

It is trivial that $\{v^i\}$ satisfies Condition \ref{cond:condition for induction of submersion propn} with $\xi=(i)$. We now establish that $\calV=\{[v^1,v^2]\}$ satisfies Condition \ref{cond:condition for induction of submersion propn}. 
\begin{proof}
We fix the compact set $K$ for which we seek to verify Condition \ref{cond:condition for induction of submersion propn} is satisfied by $\{[v^1,v^2]\}$. For $0<\epsilon<\frac{\delta}{\sup_j\lvert\lvert v^j \rvert\rvert_{\infty}}$ to be determined, we define $\xi=(1,2,1)$ (so that $\ell=\lvert \xi\rvert=3$), $U_{\gamma}:=\{(s,t,u)\in \circdelta_3^{\epsilon}: s,u\leq \gamma t\}$ for $\gamma>0$ and $L=\partial_s-\partial_u$. 

It is immediate that \eqref{eq:positive Lebesgue lim inf for induction} is satisfied, so we have only to check \eqref{eq:expression for diff operator applied to Phixi for induction proof} is satisfied for some positive function $f((s,t,u))$. We write $\unt=(s,t,u)$.

For $x\in K$ we calculate (using that $s,u=o_{\gamma\ra 0}(1)t$) that
\[
\begin{split}
\partial_u\Phi^{\xi}(x,\unt)=v^1(x+v^1(x)s+{v^2}(x)t+v^1(x)u+to_{\lvert \unt\rvert\ra 0}(1))
=v^1(x)+t(Dv^1(x){v^2}(x)+o_{\gamma\vee\lvert\unt\rvert\ra 0}(1))\\
\partial_s\Phi^{\xi}(x,\unt)=D(\varphi^{v^1}_u\circ \varphi^{v^2}_t)(\varphi_s^{v^1}(x)){v^1}(\varphi_s^{v^1}(x))=D\varphi^{v^1}_u(\varphi^{v^2}_t\circ\varphi^{v^1}_s(x))D\varphi^{v^2}_t(\varphi^{v^1}_s(x)){v^1}(\varphi^{v^1}_s(s))\\
=(\text{Id}+to_{\gamma\ra 0}(1))(\text{Id}+tD{v^2}(x)+to_{\lvert t\rvert\vee\gamma\ra 0}(1))({v^1}(x)+to_{\gamma\ra 0}(1))\\
={v^1}(x)+t(D{v^2}(x){v^1}(x)+o_{\gamma\vee\lvert\unt\rvert\ra 0}(1)).
\end{split}
\]
From this we obtain
\[
L\Phi^{\xi}(x,\unt)=t([{v^1},{v^2}](x)+o_{\gamma\vee\lvert\unt\rvert\ra 0}(1))\quad \text{for}\quad x\in K.
\]

Therefore we have \eqref{eq:expression for diff operator applied to Phixi for induction proof} with $f((s,t,u))=t$.
\end{proof}
Having established that $\{V\}$ satisfies Condition \ref{cond:condition for induction of submersion propn} for all $V\in \calV_1$, Lemma \ref{lem:union of two families of vector fields satisfies submersion condition} implies that $\calV_1$ satisfies Condition \ref{cond:condition for induction of submersion propn}, which we recall suffices to establish the existence of $U,\xi,\epsilon$ and $\delta$ providing for Part \ref{enum:submersion on open set for correct sequence of states} of Proposition \ref{prop:Quantities for Ck density thm}.

\subsubsection*{Parts \ref{enum:averaged vector field close for Q in open cone after time T} and \ref{enum:access pt of propn C infty density} of Proposition \ref{prop:Quantities for Ck density thm}}

We define $K^{Q}_t:=\bar\varphi_t^{Q}(M)$ for $t\geq 0$ and $Q\in \calQ$. Then $K^{Q_{\ast}}_t\cap \{x:d(x,K^{Q_{\ast}}\geq \frac{\delta}{2})\}$ is a descending sequence of compact sets with intersection $K^{Q_{\ast}}\cap\{x:d(x,K^{Q_{\ast}}\geq \frac{\delta}{2}\}=\emptyset$, so that for some $T$ large enough $K^{Q_{\ast}}_T\subseteq B(K^{Q_{\ast}}, \frac{\delta}{2})$. This determines $T<\infty$.

We recall that $\varpi_Q$ is the stationary distribution for the continuous-time Markov chain on $E$ with rate matrix $Q$. It follows from \cite[theorems 7 and 8]{Lax2007} that $\varpi_Q$ is a continuous function of $Q$. Therefore for some $r>0$ sufficiently small, we can ensure that $d(\bar\varphi^Q_T(x),\bar\varphi^Q_T(x^{\ast}))<\delta$ for every $Q\in B(Q^{\ast},r)$ and $x\in N$. Since the averaged vector field for $\lambda Q$ is the same as for $Q$, for all $\lambda>0$ and $Q\in \calQ$, we have that the above is true for all $Q\in \calQ_{\ast}:=\cup_{\lambda >0}\{\lambda Q':Q'\in B(Q_{\ast},r)\}$. Thus we have established the existence of $T$, $N$ and $\calQ_{\ast}$ providing for Proposition \ref{prop:Quantities for Ck density thm}.

We finally turn to Part \ref{enum:access pt of propn C infty density}. Assumption \ref{assum:accessible Hormander point} and the continuity of the flow maps imply that for all $x\in M$ there exists an open ball $B_x\ni x$, a sequence $\eta_x\in \calI$ of length $\ell_x=\lvert \eta_x\rvert$ and some $T_x<\infty$ such that for all $x'\in B_x$ there exists $\unt'\in\circdelta^{\ell_x}_{T_x}$ satisfying $\Phi^{\eta_x}_{\unt}(x')\in N$. Since $M$ is compact, we may take $x_1,\ldots,x_m$ such that $B_{x_1},\ldots,B_{x_m}$ cover $M$. We define $\eta:=\eta_{x_1}\oplus\ldots\oplus \eta_{x_m}$, which is of length $\ell=\ell_{x_1}+\ldots+\ell_{x_m}$, and $T':=T_{x_1}'+\ldots+T'_{x_m}+1$. For $x'\in B_{x_j}$ we take $\unt'_j\in \circdelta_{\ell_{x_j}}^{T_{x_j}}$ such that $\Phi^{\eta_{x_j}}_{\unt'_j}(x')\in N$. Then taking $\unt'_k\in\circdelta_{\ell_{x_k}}^{T_{x_k}}$ for $k\in \{1,\ldots,m\}\setminus \{j\}$ to be sufficiently small, we have that $\unt':=\unt'_1\oplus\ldots\oplus\unt'_j\oplus\ldots\oplus\unt'_m\in\circdelta_{\ell}^{T'}$ satisfies $\Phi^{\eta}_{\unt'}(x')\in N$.

\qed

Having established Proposition \ref{prop:Quantities for Ck density thm}, we have concluded the proof of Theorem \ref{theo:PDMPs with fast jump rates have Ck versions}.

\qed

\section{Proof of Theorem \ref{theo:PDMPs with slow jump rates have unbounded density}}\label{section:proof of unbounded density}

We fix $\Gamma$ satisfying Condition \ref{cond:sets giving rise to unbounded densities} and $Q\in\calQ$ satisfying \eqref{eq:inequality for unbounded density}.

Given $\uni=(i_1,\ldots,i_{\ell})\in E^{\ell}$ and $\unt=(t_1,\ldots,t_{\ell})\in \Rm_{\geq 0}^{\ell}$, we define 
\[
\Phi^{-\uni}(x,-\unt):=\varphi^{i_1}_{-t_1}\circ\ldots\circ\varphi^{i_{\ell}}_{-t_{\ell}}(x)\in \Rm^d,\quad x\in\Rm^d.
\]

The weak H\"ormander condition is satisfied at every $x\in \Gamma$, so it is also satisfied at every $x\in\Gamma$ if we replace the vector fields $v^1,\ldots,v^n$ with $-v^1,\ldots,-v^n$. We fix for the time being $x\in \Gamma$ and $z\in M$. We may choose $0<\epsilon_x<\frac{d(\Gamma,\partial N)}{10\sup_j\lvert\lvert v^j\rvert\rvert}$, $\xi^x\in \calI$ of length $\ell_x=\lvert \xi^x\rvert$ and $\unt_x\in \circdelta^{\ell_x}_{\epsilon_x}$ such that
\[
\circdelta_{\ell_x}^{\epsilon_x}\ni \unt\mapsto \Phi^{-\xi^x}(x,-\unt)
\]
is a submersion at $\unt=\unt_x$ by the proof of \cite[Theorem 4.4]{Benaim2015a}. Therefore, setting $y_x:=\Phi^{-\xi^x}(x,-\unt)\in N$,
\[
\circdelta_{\ell_x}^{\epsilon_x}\ni \unt\mapsto \Phi^{\xi^x}(y_x,\unt)
\]
is a submersion at $\unt=\unt_x$. For $\uni\in \calI$ we define the stopping time $\tilde{\tau}_{\uni}$ and transition kernel $\tilde{S}^{\uni}$ by
\[
\begin{split}
\tilde{\tau}_{\uni}:=\inf\{t>0:I_t\neq I_{t-}\quad\text{and}\quad I_{[t',t)}=\uni,\quad\text{for some}\quad 0<t'<t\quad\text{such that}\quad I_{t'}\neq I_{t'}\},\\ \tilde{S}^{\uni}((x,i),\cdot):=\Pm((X_{\tilde{\tau}},I_{\tilde{\tau}})\in \cdot)).
\end{split}
\]

We take $\tilde{i}\in E$ such that $\xi^x_1\neq \tilde{i}$. Proposition \ref{prop:general section smooth density propn} then provides for $\tilde{r}_x,\tilde{c}_x>0$ such that $\mu \tilde{S}_{\xi^x}(\cdot)\geq \tilde{c}_x\Leb_{\lvert_{B(x,\tilde{r}_x)\otimes \{i\}}}(\cdot)$ for all $\mu\in \calP(B(y_x,\tilde{r}_x)\times \{\tilde{i}\})$. Given $z\in M$, $y_x$ being accessible implies that there exists some $\delta_{zx},\tilde{c}_{zx}>0$ and $\eta_{zx}\in \calI$ such that
\[
\tilde{S}_{\eta_{zx}}((w,j),B(y_x,r_x)\times \{\tilde{i}\})\geq \tilde{c}_{zx}\quad \text{for all}\quad w\in B(z_x,\delta_{zx})\cap M\quad\text{and}\quad j\in E.
\]
We define the stopping time $\bar{\tau}_{zx}$ and the kernel $\bar{S}^{zx}$ to respectively be
\[
\begin{split}
\bar{\tau}_{zx}:=\inf\{t>\tau_{\eta_{zx}}:I_t\neq I_{t-}\quad\text{and}\quad I_{[t',t)}=\xi^x,\quad\text{for some}\quad \tau_{\eta_{zx}}<t'<t\quad\text{such that}\quad I_{t'}\neq I_{t'}\},\\ K_{zx}((w,k),\cdot):=\Pm((X_{\tau_{zx}},I_{\tau_{zx}})\in \cdot)).
\end{split}
\]
We have that
\[
\bar{S}_{zx}=\tilde{S}_{\eta_{zx}}\tilde{S}_{\xi^x}.
\]
Thus there exists $\bar{r}_{zx},\bar{c}_{zx}>0$ such that $\bar{S}_{zx}((w,j),\cdot)\geq \bar{c}_{zx}\Leb_{B(x,\bar{r}_{zx})\otimes \{i\}}$ for all $(w,j)\in B(z,\bar{r}_{zx})$. 

By the compactness of $\Gamma\times M$, we may take $(z_1,x_1),\ldots,(z_m,x_m)$ such that $\{B(z_j,\bar{r}_{z_jx_j})\times B(x_j,\bar{r}_{z_jx_j}):1\leq j\leq m\}$ cover $\Gamma\times M$. We define the stopping time $\tau$ by selecting $U$ uniformly from $\{1,\ldots,m\}$ and defining 
\[
\tau:=\bar{\tau}_{z_Ux_U}.
\]
We then define the kernels $S$ and $G$ by
\[
S((x,i),\cdot):=\Pm_{(x,i)}((X_{\tau},I_{\tau})\in \cdot),\quad G((x,i),\cdot):=\expE_{(x,i)}\Big[\int_0^{\tau}\delta_{(X_s,I_s)}ds\Big].
\]

Therefore there exists a non-negative continuous function $\phi$ on $M$ which is strictly positive on $\Gamma$ (hence uniformly bounded away from $0$ on $\Gamma$ since $\Gamma$ is compact) such that $\mu S(dx\times\{i\})\geq \phi(x)dx$ for all $\mu\in\calP(M\times E)$. Thus $S$ satisfies Doeblin's criterion, hence it has a unique stationary distribution and we can apply Proposition \ref{prop:relation pi tilde with omega tilde G tilde}. Since $\omega=\omega S$, this implies that $\omega (dx\times \{i\})\geq \phi(x)dx$. We abuse notation by writing $\phi$ both for the function $\phi$ and the measure with density $\phi$ on $M\times \{i\}$ (and no mass on the other copies of $M$).

We now define the stopping time $\tau_0$ and kernels $(G_t)_{0\leq t\leq \infty}$ by
\[
\tau_0:=\inf\{t>0:I_{t}\neq I_{t'}\},\quad G_t((x,j),\cdot):=\expE_{(x,j)}\Big[\int_0^{\tau_0\wedge t}\delta_{(X_s,I_s)}(\cdot)ds\Big],\quad 0\leq t<\infty.
\]
Then since $\tau_0\leq \tau$ by construction, $G_t\leq G_{\infty}\leq G$ for all $t<\infty$. Therefore
\[
\expE_{\omega}[\tau]\pi=\omega G\geq \omega G_{\infty}\geq \phi G_t\quad \text{for all}\quad t<\infty.
\]
We can calculate by the change of variables formula that $\phi G_t$ has a density (which we label $\rho_t$) on $M\times \{i\}$ given by
\[
\rho_t(y)=\int_0^{t}e^{Q_{ii}s}\text{det}(D_y\varphi^i_{-s}(y))\phi(\varphi^i_{-s}(y))ds=\int_0^{t}e^{Q_{ii}s}e^{\int_0^s\text{tr}(Dv^i(\varphi^i_{-r}(y))dr}\phi(\varphi^i_{-s}(y))ds.
\]
Since $\Gamma$ is backward invariant under $\varphi^i$, we observe that 
\[
\rho_t(y)\geq \int_0^{t}e^{(Q_{ii}+R(\Gamma,i))s}\phi(\varphi^i_{-s}(y))ds\underbrace{\geq}_{\text{by }\eqref{eq:inequality for unbounded density}} t\inf_{x\in \Gamma}\phi(x)\quad\text{for}\quad y\in \Gamma.
\]
Therefore $(\rho_t)_{t=1}^{\infty}=(\phi G_t)_{t=1}^{\infty}$ is a non-decreasing sequence of non-negative continuous functions, which converge to $+\infty$ on $\Gamma$. This implies that $\pi$ is bounded from below by a lower semicontinuous function which is everywhere infinite on $\Gamma\times\{i\}$.

{\textbf{Acknowledgement:}}  This work was funded by grant 200020 196999 from the Swiss National Foundation.

\qed
\bibliography{Library.tex}
\bibliographystyle{plain}
\end{document}